\documentclass[10pt,a4paper]{article}

\usepackage[english]{babel}
\usepackage[utf8x]{inputenc}
\usepackage[T1]{fontenc}
\usepackage{amssymb,amsmath}

\usepackage{booktabs}  %for table creation
\usepackage{multirow,amsthm}
\usepackage{rotating}
\usepackage{epsf,epstopdf}
\usepackage{graphicx,graphics}
\usepackage{color}
\usepackage[skip=2pt,font = scriptsize]{caption}
\usepackage{float}
\usepackage{adjustbox}
\newtheorem{obs}{Observation}
\newtheorem{alg}{Algorithm}

\def\vs{\vspace}
\newtheorem{theorem}{Theorem}[section]
\newtheorem{lemma}[theorem]{Lemma}

\newtheorem{example}{Example}

\usepackage[a4paper,top=3cm,bottom=2cm,left=3cm,right=3cm,marginparwidth=1.75cm]{geometry}
\usepackage{fancyhdr,ifthen}
\usepackage{amsmath,authblk,lipsum}
\usepackage{graphicx}
\usepackage[colorinlistoftodos]{todonotes}
\usepackage[colorlinks=true, allcolors=blue]{hyperref}
\title{A modification of the Jacobi-Davidson  method
\thanks{Research supported by CSIR India(09/084(0563)/2010
EMR-I) and NBHM India(2/40(3)/2016/R\&D-II/9602)}}
\author{Mashetti Ravibabu\\Department of Computational and Data Sciences,\\ Indian Institute of Science, Bengaluru,\\India-560012.\thanks{mashettiravibabu2@gmail.com}}
\begin{document}
\maketitle
\noindent Each iteration in Jacobi-Davidson method for solving large sparse eigenvalue problems involves two phases, called subspace expansion and eigen pair extraction. The subspace expansion phase involves solving a correction equation. We propose a modification to this by introducing a related correction equation, motivated by the least squares. We call the proposed method as the Modified Jacobi-Davidson Method. When the subspace expansion is ignored as in the Simplified Jacobi-Davidson Method, the modified method is called as Modified Simplified Jacobi-Davidson Method. We analyze the convergence properties of the proposed method for Symmetric matrices. Numerical experiments have been carried out to check whether the method is computationally viable or not.\\

\noindent \textbf{keywords:}Jacobi-Davidson method, Subspace expansion, Eigen pair extraction
\section{Introduction}
Projection methods are widely used for solving large sparse eigenvalue problems. Jacobi-Davidson
method is a quite well known projection method that approximates the smallest eigenvalue or eigenvalue
near a given shift of a symmetric matrix. This method starts with an arbitrary initial nonzero vector and involves two phases. In the first phase, the eigenvector and eigenvalue approximations are obtained from the 
existing subspace by applying a projection to the given matrix.  From these approximate eigenvalues, we select the one with desired properties  and its corresponding eigenvector approximation. The second phase involves solving the correction equation to obtain a vector from the orthogonal complement of the selected eigenvector
approximation. An existing subspace is expanded by adding this vector after orthogonalizing it against the existing subspace. In the proposed method, we modify the correction equation using least squares. A vector is determined from the orthogonal complement of a selected eigenvector approximation so that the norm of the residual associated with  the resultant of the selected eigenvector approximation and the vector determined from its orthogonal complement is minimum.
 
In Section 2, we briefly review Jacobi-Davidson method. In Section 3, a new correction equation is introduced based on least squares heuristics, and convergence properties of the proposed method are discussed. In Section 4, we report the results of numerical experiments which were carried out to understand the viability of the method developed in Section 3. In Section~5, we consider the method used with restarting, and report the results of numerical examples. Section~6 concludes the paper. 

\section{Jacobi-Davidson and simplified Jacobi-Davidson method}
We briefly describe Jacobi-Davidson method. Let $A$ be a given matrix. Assume that we have already computed a
matrix $V_k$ with $k$ orthonormal columns, which span the existing subspace of dimension $k$, starting with an initial nonzero vector $v_1$. Approximations to eigenvalues and eigenvectors of $A$ are obtained from those of the matrix $H_k = V_k^\ast A V_k$ as follows. If $(\theta,y)$ is an eigen pair of $H_k$, then $\theta$ is called a Ritz value and $u = V_ky$ is called a Ritz vector, and the associated residual is $r = Au-\theta u$. Since $V_k$ has orthonormal columns, $V_k^\ast r = V_k^\ast AV_ky -\theta V_ky =0$. Thus, the residual is orthogonal to the existing subspace. We then select one of the Ritz values with desired properties and its corresponding Ritz vector. Using this vector, we solve the following correction equation, for the vector $t$:
\begin{equation}\label{equn1}\relax
(I-uu^\ast )(A-\theta I) (I-uu^\ast)t = -r,~~ t \perp u
\end{equation}
\noindent The vector $t$ is orthogonalized with respect to the existing subspace of dimension $k$ spanned by $k$ orthonormal columns of the matrix $V_k$ and then normalized to get the vector $$v_{k+1} = \frac{(I-V_kV_k^\ast
)t}{\|(I-V_kV_k^\ast )t\|}$$ Then $v_{k+1}$ is added to the existing subspace of
dimension $k$ and $V_k$ is updated to $V_{k+1} = [V_k~~, v_{k+1}]$. An eigen pair of $H_{k+1}$ is then taken as an improved eigen pair. Since $A$ is very large and sparse, Equation (\ref{equn1}) is solved using
a Krylov subspace method for linear systems, such as GMRES or FOM. The algorithm is continued until the norm of the residual satisfies some fixed tolerance, that is, $\|r|\|_2 \leq tol$.\\
\paragraph*{} In this process, if we ignore the subspace expansion, then the method is called the Simplified Jacobi-Davidson method. Here, for an eigenvector approximation $u$  and the solution $t$ of the correction
equation (\ref{equn1}), the vector $\frac{u+t}{\|u+t\|}$ is considered as  the new eigenvector approximation,  The process is continued by replacing $u$ and $\theta$ in the correction equation (\ref{equn1}), with $\frac{u+t}{\|u+t\|}$ and the Rayleigh quotient of the vector $\frac{u+t}{\|u+t\|},$ respectively. The process
terminates when norm of the residual vector satisfies a pre-set tolerance. It differs from Jacobi-Davidson method by ignoring the information from previous vectors, where as in Jacobi-Davidson method, a subspace is formed from these vectors.

\section{Modification to the Simplified Jacobi-Davidson method}
As we have seen in Section 2, Simplified Jacobi Davidson method uses at each step, the solution of the
correction equation (\ref{equn1}),  where $u$ is a known eigenvector approximation, $\theta $ is
its Rayleigh quotient with respect to the matrix $A,$ and the vector $$\frac{u+(I-uu^\ast)t}{\|u+(I-uu^\ast)t\|}$$ is the new eigenvector approximation. It is a general belief that the matrix $(I-uu^\ast )(A-\theta I) (I-uu^\ast)$ is well-conditioned compared to  the matrix $(A-\theta I);$ but this is not correct as shown in the following Example.
\begin{example}\label{example1}\relax
Consider the matrix
$$
A = \begin{bmatrix}
    1 &0 &0&0 \\
    0 &0 &2&0 \\
    0 &2 &0&0 \\
    0 &0 &0&1 \\
    \end{bmatrix}
$$
Take the vector $u$ as $(0~0~1~0)^{t}$. Then $\theta = u^\ast Au = 0,$ $r = (A-\theta I)u = (0~2~0~0)^{t},$ and 
\begin{equation*}\label{equn42}\relax
B: = (I-uu^\ast )(A-\theta I) (I-uu^\ast) = \begin{bmatrix}
    1 &0 &0&0\\
    0 &0 &0&0\\
    0 &0 &0&0\\
    0 &0 &0&1\\
    \end{bmatrix}
\end{equation*}

Here, $A-\theta I$ is non-singular but the matrix $B$ is singular. Further, $-r$ is not in the range space of $B$. So there exists no vector $t$ satisfying the correction equation (\ref{equn1}).
\end{example}
 Example~1 illustrates the situation  where the correction equation has  no solution. To avoid this, we 
propose the following method of choosing a new eigenvector approximation. We call this method as the \textit{Modified Simplified Jacobi Davidson} method (MSJD). It has the following two steps. \begin{alg}\label{alg2}\title{\textbf{MSJD}}\relax
\textsf{Step 1: For a given eigenvector approximation $u $ of a matrix $A$, find
a vector $t$ that minimizes $$\|(A-\theta I)u+(A-\theta I)(I-uu^\ast
)t\|^2$$ where $\theta$ is the Rayleigh quotient of $u$ with respect to the matrix $A$. \\
Step 2: Take the new eigenvector approximation as ~~$\dfrac{u+(I-uu^\ast
)t}{\|u+(I-uu^\ast )t \|}$.}
\end{alg}
In Step 1, the vector $t$ is determined by solving the following normal equation:
\begin{equation}\label{equn43}\relax
(I-uu^\ast )(A-\theta I)^\ast (A-\theta I) (I-uu^\ast)t =
-(I-uu^\ast )(A-\theta I)^\ast (A-\theta I)u
\end{equation}
Notice that, if $(A-\theta I)$ is symmetric, then (\ref{equn43}) can be obtained from (\ref{equn1}) by replacing the matrix $(A-\theta I)$ with $(A-\theta I)^2,$ and then applying the orthogonal projection $(I-uu^\ast )$ to both the sides. 

The linear system ~(\ref{equn43}) may be solved by well known methods such as Gaussian elimination, Conjugate gradients, and etc. When $\theta $ is very close to an eigenvalue $\lambda $ of $A$, the matrix $(I-uu^\ast )(A-\theta I)^2(I-uu^\ast )$ is expected to be more ill-conditioned than $(I-uu^\ast )(A-\theta I)(I-uu^\ast )$. Such ill-conditioned problems can be solved by using well known regularization techniques, for example, Tikhonov Regularization.  Tikhonov regularization method, instead of solving the normal equations,  requires the solution of the following equation: 
\begin{equation*}\label{equn44}\relax
(I-uu^\ast )(A-\theta I)^\ast (A-\theta I) (I-uu^\ast)t  +h^2 t =
-(I-uu^\ast )(A-\theta I)^\ast (A-\theta I)u
\end{equation*}
where $h^2$ is a parameter chosen in such a way that as $h$ tends to zero, the  solution of the above equation converges to the solution of the least squares problem in Step 1. This is equivalent to finding the vector that
minimizes the functional 
\begin{equation*}\label{equn45}\relax
\|(A-\theta I)u+(A-\theta I)(I-uu^\ast )t\|^2+h^2\|t\|^2
\end{equation*}

Here we propose a new way, which avoids problems associated with parameter selection. For this purpose, we rewrite Equation~(\ref{equn43}) as 
\begin{equation}\label{equn46}\relax
(I-uu^\ast )(A-\theta I)^\ast (A-\theta I)(u+(I-uu^\ast )t) = 0
\end{equation}

When $\theta$ is close to an eigenvalue of $A$, the system matrix in Equation (\ref{equn46}) is ill-conditioned. To avoid ill-conditioning, perturb the matrix \linebreak 
$(A-\theta I)^\ast (A-\theta I)$ by a matrix $E$, and write the perturbed equation as 
\begin{equation}\label{equn46a1}\relax
(I-uu^\ast )\big((A-\theta I)^\ast (A-\theta
I)+E\big)(u+(I-uu^\ast)t) =  0
\end{equation}
We require that $\|(A-\theta I)u+(A-\theta
I)(I-uu^\ast )t\|$ be minimum. Using Equation~(\ref{equn43}), we simplify (\ref{equn46a1}) to obtain
\begin{equation}\label{equn46a2}\relax
(I-uu^\ast)E(u+(I-uu^\ast)t = 0
\end{equation}
 Matrices that transform the vector $u+(I-uu^\ast )t$ to a vector parallel to $u$ are possible candidates for $E$ in Equation~(\ref{equn46a2}). Also,  $E ~:=~ kuw^\ast$  for any scalar $k$ and any vector $w,$ satisfies (\ref{equn46a2}). For the choice $w = u$, the matrix $E$ becomes symmetric and it may have computational advantages. However this simple choice is not possible, in general. 
\begin{obs}\label{obs1}\relax
The choice $w = u$ is not possible unless $(\theta, (I-uu^\ast )t)$ is an eigen pair of $A$ or $u$ is a right singular vector of $(A-\theta I)$.
\end{obs} 
\begin{proof}
Rewrite Equation~(\ref{equn46a1}) as
$$\big((A-\theta I)^\ast (A-\theta I)+E\big)(u+(I-uu^\ast )t) = k_1u$$ 
where $k_1 := u^\ast\big( (A-\theta I)^\ast (A-\theta I)+E\big)(u+(I-uu^\ast )t)$ is a scalar. Taking the inner product with $u+(I-uu^\ast )t,$ we obtain
\begin{eqnarray*} 
 \|(A-\theta I)(u+(I-uu^\ast)t)\|^2+(u+(I-uu^\ast )t)^\ast E(u+(I-uu^\ast)t) \\ 
 =  k_1(u+(I-uu^\ast )t)^\ast u
\end{eqnarray*}
Since $\|u\| = 1$, this equation is simplified to
\begin{equation}\label{equn46a3}\relax
\|(A-\theta I)(u+(I-uu^\ast)t)\|^2 = -(u+(I-uu^\ast )t)^\ast E(u+(I-uu^\ast)t)+k_1
\end{equation} 
In order that the left hand side of Equation~(\ref{equn46a3}) is minimum over all vectors $t$, the gradient of the right hand side of Equation~(\ref{equn46a3}) with respect to the vector $t$ vanishes. Therefore, 
\begin{equation*}\label{equn46a4}\relax
-2(I-uu^\ast)E(u+(I-uu^\ast)t+(I-uu^\ast ) \big((A-\theta I)^\ast (A-\theta I)+E\big )u = 0
\end{equation*}
Using Equation (\ref{equn46a2}), we have $(I-uu^\ast ) \big((A-\theta I)^\ast (A-\theta I)+E\big )u = 0$. This together with Equation (\ref{equn46a1}) gives 
\begin{equation*}\label{equn46b5}\relax
(I-uu^\ast )\big((A-\theta I)^\ast (A-\theta I)+E\big )(I-uu^\ast )t = 0
\end{equation*}
If $E$ is of the form $kuu^\ast$, then $(I-uu^\ast )(A-\theta I)^\ast (A-\theta I)(I-uu^\ast )t = 0$. This implies $$\|(A-\theta I)(I-uu^\ast )t \| = 0$$ Therefore, the choice $w=u$ is not possible, unless $(\theta,(I-uu^\ast )t)$ is an eigen pair of $A$ or $(I-uu^\ast)t = 0$. From Equation~(\ref{equn46}), it is clear that if $(I-uu^\ast ) t = 0,$ then $u$ is a right singular vector of $(A-\theta I)$.
\end{proof}
\begin{obs}\label{obs2}\relax
If $E=kuw^\ast$ satisfies Equation~(\ref{equn46a2}), then the component of $w$ orthogonal to $u$ is parallel to the vector $\big((A-\theta I)^\ast(A-\theta I) - \|(A-\theta I)u\|^2 I\big)u$.
\end{obs}
\begin{proof}
Vanishing of the gradient of right hand side of Equation~(\ref{equn46a3}) gives
\begin{equation*}\label{equn46a5}\relax
-(I-uu^\ast)(E+E^\ast )\big(u+(I-uu^\ast)t\big)+(I-uu^\ast)\big((A-\theta I)^\ast (A-\theta I)+E\big)u = 0
\end{equation*}
Using Equation~(\ref{equn46a2}), we have
 $$-(I-uu^\ast)E^\ast \big(u+(I-uu^\ast)t\big)+(I-uu^\ast)\big((A-\theta I)^\ast (A-\theta I)+E\big)u = 0$$
With $E=kuw^\ast$ and $\|u\| = 1$, we have
\begin{equation*}\label{equn46a6}\relax
-k(I-uu^\ast)w+(I-uu^\ast)(A-\theta I)^\ast (A-\theta I)u = 0
\end{equation*}
It gives
$$-k(I-uu^\ast)w+(A-\theta I)^\ast(A-\theta I)u = \|(A-\theta I)u\|^2 u$$ Therefore, $k(I-uu^\ast )w = \big((A-\theta I)^\ast(A-\theta I) - \|(A-\theta I)u\|^2 I\big)u$. 
\end{proof}

Observation-\ref{obs2} gives the direction of the component of $w$ orthogonal to the vector $u.$ Moreover, it is clear that the component of $w$ in the direction of $u$ can be chosen as $\alpha u$, for some scalar $\alpha$. Therefore, the vector $w$ in Observation-\ref{obs2} is of the form $w=\alpha u+\big((A-\theta I)^\ast(A-\theta I) - \|(A-\theta I)u\|^2 I\big)u.$ A scalar $\alpha$ can be chosen by imposing an extra condition on the vector $w$ such as $\|w\|_2=1.$ Further, notice that when $E=kuw^\ast,$ solutions of Equations (\ref{equn46}) and (\ref{equn46a1}) coincide. Then Equation~(\ref{equn46}) can be rewritten as
\begin{equation}\label{equn46ab2}\relax
\big(~(A-\theta I)^\ast (A-\theta I)\big)(u+(I-uu^\ast)t) = k_1u
\end{equation} 
with $k_1 = u^\ast (A-\theta I)^\ast(A-\theta I)(u+(I-uu^\ast)t)$.
If $A$ is symmetric, then from Equation~(\ref{equn46ab2}) we have  
$u+(I-uu^\ast)t =k_1(A-\theta I)^{-2}u$. In this case, the new
eigenvector approximation is  
\begin{equation}\label{newapprox}\relax
\frac{(A-\theta
I)^{-2}u}{\|(A-\theta I)^{-2}u\|}
\end{equation}
Next, we deal with the convergence of the norms of residual vectors in MSJD method.
\begin{theorem}\label{thm1c}
Let $u_k$ denote an approximation to an eigenvector $x$ corresponding to the eigenvalue $\lambda$ of $A$ at the $k$th iteration of the MSJD method. Let  $\rho_k:= \rho(u_k)$ be the  Rayleigh quotient of $u_k$ with respect to the matrix $A$. Then the sequence of residual norms $\{\|(A-\rho_k I)u_k\|^2\}$  and  the sequence $\{\|(A-\rho_k I)u_{k+1}\|^2\}$ converge to the same limit. Further the sequence of absolute differences $|\rho_{k+1}-\rho_k|$ between Rayleigh quotients in two consecutive iterations of MSJD method converges to $0$ as $k \rightarrow \infty.$ 
\end{theorem}
\begin{proof}
From Step~1 of MSJD method, it follows that
\begin{equation}\label{equn46b}\relax
\|(A-\rho_k I)(u_k+(I-u_ku_k^\ast )t_k)\|^2 \leq \|(A-\rho_k
I)u_k\|^2
\end{equation}
And from Step 2, we have 
\begin{equation}\label{equn46ab1}\relax
u_{k+1} = \frac{u_k+(I-u_ku_k^\ast
)t_k}{(1+\|(I-u_ku_k^\ast )t_k\|^2)^\frac{1}{2}}
\end{equation}
Notice that $\|u_k\| = 1$ and  $\|u_k+(I-u_ku_k^\ast )t_k\|^2 = 1+\|(I-u_ku_k^\ast )t_k\|^2 \geq 1.$ Then Equations (\ref{equn46b}) and (\ref{equn46ab1}) give  
\begin{eqnarray}\label{equn46b1}\relax
\|(A-\rho_k I)u_{k+1}\|^2 = \frac{\|(A-\rho_k I)(u_k+(I-u_ku_k^\ast
)t_k)\|^2}{1+\|(I-u_ku_k^\ast )t_k\|^2} \nonumber \\ \leq \|(A-\rho_k
I)(u_k+(I-u_ku_k^\ast )t_k)\|^2
\end{eqnarray}
As $\|u_{k+1}\| = 1 $ and  $\rho_{k+1}$ is a Rayleigh quotient of the vector $u_{k+1}$ with respect to $A$, we have $\langle (A-\rho_k I)u_{k+1},(\rho_{k+1}-\rho_k)u_{k+1} \rangle = (\rho_{k+1}-\rho_k)^2.$ Thus,
\begin{eqnarray}\label{equn46b2}\relax
\|(A-\rho_{k+1} I)u_{k+1}\|^2 = \|\big((A-\rho_k I)-(\rho_{k+1}-\rho_k)I\big)u_{k+1}\|^2  \nonumber \\= \|(A-\rho_k I)u_{k+1}\|^2-|\rho_{k+1}-\rho_k|^2~~~~
\end{eqnarray}
Therefore $\|(A-\rho_{k+1} I)u_{k+1}\|^2  \leq \|(A-\rho_k I)u_{k+1}\|^2$.  Combining this inequality with Equations~(\ref{equn46b1}) and (\ref{equn46b2}), we obtain
\begin{equation*}\label{equn46c}\relax
\|(A-\rho_{k+1} I)u_{k+1}\|^2  \leq \|(A-\rho_k I)u_{k+1}\|^2 \leq
\|(A-\rho_k I)u_k\|^2
\end{equation*}
It shows that $\{\|(A-\rho_k I)u_k\|^2\}$ is a
monotonically decreasing sequence of non-negative real numbers. Suppose it converges to $\alpha$. By Sandwich theorem,  the sequence $\{\|(A-\rho_k I)u_{k+1}\|^2\}$ also converges to $\alpha$. Using Equation~(\ref{equn46b2}), we conclude that $|\rho_{k+1}-\rho_k| \rightarrow 0$ as $k \rightarrow \infty.$
\end{proof}
Though $|\rho_{k+1}-\rho_k| \rightarrow 0$ is enough for computational purposes, it does not imply that $\{\rho_k\}$ is a convergent sequence. In order to prove a result on convergence alternatives, we use the following two lemmas.
\begin{lemma}\label{thm1b}\relax
Let $u$  be unit vector. Let $\alpha $ be the Rayleigh quotient of $u$ with respect to matrix $A$. Let $s := u+ \tau (I-uu^\ast )t$, where $\tau$ is chosen such that Rayleigh quotient of $s$ is minimum over the space of vectors spanned by $u$ and $(I-uu^\ast ) t.$ Write $J_{u,s}: = (I-uu^\ast ) (A-\rho(s)I) (I-uu^\ast .)$ Then the following relationship holds: 
\begin{equation}\label{equn18}\relax
\langle J_{u,s}(u-s),u-s \rangle = \rho(u) - \rho(s)
\end{equation}
\end{lemma}

For the proof of Lemma~\ref{thm1b}, see Theorem 2.6 in \cite{ovtline}. 

\begin{lemma}\label{thm3}\relax
Let $\rho_k$, $u_k$ and $t_k$ be as in Theorem~(\ref{thm1v}). Let $\tau$ be such that for the vector $s := u_k+ \tau (I-u_ku_k^\ast )t_k$  norm of the residual $\frac{\|(A-\rho_k I)s\|}{\|s\|}$ is minimum over the subspace spanned by $u_k$ and $(I-u_ku_k^\ast )t_k$. Then 
\begin{equation}\label{equn32}\relax
\frac{\|(A-\rho_k I)(I-u_ku_k^\ast )t_k\|^2}{\|(I-u_ku_k^\ast )t_k\|^2} \geq
\|(A-\theta I)u_k\|^2
\end{equation}
and
\begin{equation}\label{equn32a}\relax
\|(I-u_ku_k^\ast)t_k\|^2 \leq 1
\end{equation}
\end{lemma} 
\begin{proof}
 For the vectors $u_k$ and $s,$ we have 
 $$J_{u_k,s} = (I-u_ku_k^\ast )\big( (A-\rho_k I)^\ast (A-\rho_k I)-\frac{\|(A-\rho_k I)s\|^2}{\| s \|^2} I \big ) (I-u_ku_k^\ast )$$ 
 and $u_k-s = -\tau (I-u_ku_k^\ast )t_k$. Using Equation~(\ref{equn18}) for the matrix \linebreak
 $ (A-\rho_k I)^\ast (A-\rho_k I)$ and the vectors $u_k$ and $s,$ we obtain 
\begin{eqnarray*}\label{equn30}\relax
\tau^2 \big(\|(A-\rho_k I)(I-u_ku_k^\ast )t_k\|^2-\frac{\|(A-\rho_k
I)s\|^2}{\| s \|^2} \cdot \|(I-u_ku_k^\ast )t_k\|^2\big) \\ =
\|(A-\rho_k I)u_k\|^2-\frac{\|(A-\rho_k I)s\|^2}{\| s \|^2}
\end{eqnarray*}
It follows that
\begin{equation}\label{equn31}\relax
\frac{\|(A-\rho_k I)s\|^2}{\| s \|^2} = \frac{\|(A-\rho_k
I)u_k\|^2-\tau^2 \|(A-\rho_k I)(I-u_ku_k^\ast )t_k\|^2}{1-\tau^2
\|(I-u_ku_k^\ast )t_k\|^2}
\end{equation}
Since $\frac{\|(A-\rho_k I)s\|^2}{\| s \|^2} \leq \|(A-\rho_k
I)u_k\|^2$,  Equation~(\ref{equn31}) yields the inequality in Equation~(\ref{equn32}). From Equation~(\ref{equn46b}), we have 
$$\|(A-\rho_k I)u_k\|^2 \geq \|(A-\rho_k I)(I-u_ku_k^\ast )t_k\|^2.$$ 
It gives $\|(I-u_ku_k^\ast )t_k\|^2 \leq 1$. 
\end{proof}
\begin{theorem}\label{thm1v}\relax
Let the scalar $\rho_k$ and the vector  $u_k$ be as in Theorem~\ref{thm1c}.  Then the sequence of eigenvector approximations in MSJD method converges to either an eigenvector of the matrix $A$ corresponding to the eigenvalue $\rho,$  or to a vector in an invariant subspace corresponding to eigenvalue $\rho,$ or to a vector in an invariant subspace spanned by eigenvectors corresponding to eigenvalues $\rho \pm \alpha.$
\end{theorem}
\begin{proof}
We see that 
\begin{eqnarray*}
\|(A-\rho_k I)(u_k+(I-u_ku_k^\ast )t_k)\|^2 = \|(A-\rho_k I)u_k\|^2+\|(A-\rho_k I)(I-u_ku_k^\ast )t_k)\|^2 \\ 
-2\operatorname{Re}(\langle (A-\rho_k I)u_k, (A-\rho_k I)(I-u_ku_k^\ast )t_k) \rangle ) 
\end{eqnarray*}
Taking inner product with $t_k$ on both sides of equation (\ref{equn43}), we get
$$\langle (A-\rho_k
I)u_k, (A-\rho_k I)(I-u_ku_k^\ast )t_k) \rangle = -\|(A-\rho_k I)(I-u_ku_k^\ast )t_k)\|^2$$
Therefore,
\begin{equation*}\label{equn46e}\relax
\|(A-\rho_k I)(u_k+(I-u_ku_k^\ast )t_k)\|^2 = \|(A-\rho_k
I)u_k\|^2-\|(A-\rho_k I)(I-u_ku_k^\ast )t_k)\|^2
\end{equation*}
Using Equation (\ref{equn46ab1}), we obtain
\begin{equation*}\label{equn46g1}\relax
\|(A-\rho_k I)u_{k+1}\|^2 = \frac{\|(A-\rho_k
I)u_k\|^2-\|(A-\rho_k I)(I-u_ku_k^\ast )t_k)\|^2}{1+\|I-u_ku_k^\ast )t_k\|^2}
\end{equation*}
It implies that
\begin{eqnarray}\label{equn46g}\relax
\|(A-\rho_k I)u_k\|^2-\|(A-\rho_k I)u_{k+1}\|^2 ~~~~~~~~~~~~~~~~~~~~~~~~~~~~~~~~~~ \nonumber\\  = 
\frac{\|(A-\rho_k I)u_k\|^2\|I-u_ku_k^\ast )t_k\|^2+\|(A-\rho_k
I)(I-u_ku_k^\ast )t_k)\|^2}{1+\|I-u_ku_k^\ast )t_k\|^2}
\end{eqnarray}
 Due to Theorem~\ref{thm1c}, the left hand side of Equation~(\ref{equn46g}) converges to zero as $k \rightarrow \infty$. Therefore $$\|(A-\rho_k I)u_k\|^2\|I-u_ku_k^\ast )t_k\|^2+\|(A-\rho_k
I)(I-u_ku_k^\ast )t_k)\|^2 \rightarrow 0$$ 
From Equation (\ref{equn32}), we have
\begin{eqnarray*}
\|(A-\rho_k I)u_k\|^2\|(I-u_ku_k^\ast )t_k\|^2+\|(A-\rho_k
I)(I-u_ku_k^\ast )t_k)\|^2 \\~~~~~~~
\geq 2\|(A-\rho_k I)u_k\|^2\|I-u_ku_k^\ast )t_k\|^2
\end{eqnarray*}
 Therefore, $\|(A-\rho_k I)u_k\|^2\|(I-u_ku_k^\ast )t_k\|^2 \rightarrow 0.$ Since $\|(A-\rho_k I)u_k\|^2$ is a monotonically decreasing sequence of nonnegative terms, it converges to a non-negative real number. Suppose it converges to $\alpha^2 \geq 0.$ We have two cases:$$\|(I-u_ku_k^\ast )t_k\|^2 \rightarrow 0~\quad \mbox{or}~ \quad \|(A-\rho_k I)u_k\|^2 \rightarrow 0 $$
Note that in both the cases, Equation~(\ref{newapprox}) can be written as
\begin{equation}\label{equn46h}\relax
(A-\rho_k I)^2 u_{k+1} = \frac{\|(A-\rho_k
I)u_{k+1}\|^2}{u_{k+1}^*u_k}u_k
\end{equation}

%\noindent \textbf{Case 1:} $\|(I-u_ku_k^\ast )t_k\|^2 \rightarrow 0$ and the sequence $\{u_k\}$ converges, say, to $u.$ \\
%In this case, the vectors $u_{k+1}$ and $u_k$ are nearly parallel as $k \rightarrow \infty.$  Since $\|(A-\rho_k I)u_{k+1}\|^2$ converges to $\alpha^2$, Equation~(\ref{equn46h}) gives 
%\begin{equation}\label{equn46i}\relax
%(A-\rho I)^2 u = \alpha^2 u
%\end{equation}
%This situation is similar to Rayleigh quotient iteration. In this case, $\rho$ is equal to either $\lambda_i +\alpha$ or $\lambda_i-\alpha$ for one or more eigenvalues of $A$. 
\noindent \textbf{Case 1:} $\|(I-u_ku_k^\ast )t_k\| \rightarrow 0$ \\\\
From Step~2 in Algorithm~\ref{alg2}, we have
$$u_{k+1} = \frac{u_k+(I-u_ku_k^\ast )t_k}{\|u_k+(I-u_ku_k^\ast )t_k\|}$$
Since $u_k$ is orthogonal to $(I-u_ku_k^\ast )t_k$ and $\|u_k\| = 1$, we have
$$\|u_k+(I-u_ku_k^\ast )t_k\|^2 = \|u_k\|^2+\|(I-u_ku_k^\ast )t_k\|^2 = 1+\|(I-u_ku_k^\ast )t_k\|^2.$$
Writing $\alpha_k = \sqrt{1+\|(I-u_ku_k^\ast )t_k\|^2},$ we see that
$u_{k+1}-\frac{u_k}{\alpha_k} = \frac{(I-u_ku_k^\ast )t_k}{\alpha_k}.$ Since $\|(I-u_ku_k^\ast )t_k\| \rightarrow 0,~\alpha_k \rightarrow 1$ as $k\to\infty.$ And then $u_{k+1}-u_k \rightarrow 0.$ From  Equation~(\ref{equn46h}), we thus obtain 
$$(A-\rho_k I)^2(u_{k+1}-u_k)+ (A-\rho_k I)^2u_k = \frac{\|(A-\rho_k I)(u_{k+1}-u_k)+ (A-\rho_k I)u_k\|^2}{u_{k+1}^*u_k}$$
As $u_{k+1}-u_k \rightarrow 0,$ 
$(A-\rho_k I)^2u_k = \alpha^2u_k ~~\mbox{for large}~~k.$ Then
$$[(A-\rho_k I)^2 - \alpha^2 I]u_k = 0 ~~\mbox{for large}~~k.$$
Therefore, In this case also, $\rho_k$ is equal to either $\lambda_i +\alpha$ or $\lambda_i-\alpha$ for one or more eigenvalues of $A.$ \\

\noindent \textbf{Case 2:} $\|(A-\rho_k I)u_k\|^2 \rightarrow 0$ and the sequence $\{u_k\}$ converges
to $u.$ Then $u$ is an eigenvector corresponding to $\rho.$ \\

\noindent \textbf{Case 3:}  $\|(A-\rho_k I)u_k\|^2 \rightarrow 0$ and the sequence $\{u_k\}$ is not
convergent. \\\\
Then Equation~(\ref{equn46ab1}) gives
$$\frac{1}{u_{k+1}^\ast u_k} = (1+\|(I-u_ku_k^\ast )t_k\|^2)^{\frac{1}{2}}$$ From Equation~(\ref{equn32a}), we have $\|(I-u_ku_k^\ast )t_k\|^2 \leq 1$.  This together with Equation ~(\ref{equn46h}) gives
$$\|(A-\rho_k I)^2u_{k+1}\| \leq \sqrt{2}\|(A-\rho_k I)u_{k+1}\|^2u_k$$ which implies that $\|(A-\rho_k I)^2u_{k+1}\| \rightarrow 0$. Then, $\|(A-\rho_k I)u_k\|^2\to 0$ and $\|(A-\rho_k I)^2u_{k+1}\|\to 0$ as $k \rightarrow \infty.$ It follows that, as $k \rightarrow \infty,$  $u_k$ is in the invariant subspace corresponding to the eigenvalue $\rho$ .
\end{proof}
 
We show that if $\{u_k\}$ converges to an eigenvector corresponding to the eigenvalue $\lambda$, then the order of convergence is $5$.  The proof follows the line of proof of Theorem~4.7.1 in \cite{par}.

\begin{theorem}\label{rq}\relax
Let the sequence $\{u_k\}$ generated by MSJD method converge to an eigenvector $x$ of $A$ corresponding to the eigenvalue $\lambda$. Let $\Phi_k$ denote the angle between vectors $u_k$ and $x$. Then $|\Phi_{k+1}| \leq |\Phi_k|^5$ for large $k.$
\end{theorem}
\begin{proof}
Since $\Phi_k$ is the angle between the unit vectors $u_k$ and $x$, $u_k$ can be written as
\begin{equation}\label{equn47}\relax
u_k = x\cos\Phi_k +v_k\sin\Phi_k  
\end{equation}
Where $v_k^\ast x = 0$ and $\|v_k\|= \|x\|=1.$ Let $\rho_k:= \rho(u_k)$ be the Rayleigh quotient of $A$ with respect to $u_k$. Pre-multiply  Equation~(\ref{equn47}) with  $(A-\rho_kI)^{-2}$ and use Spectral mapping theorem to obtain  
\begin{equation}\label{equn48}\relax
(A-\rho_kI)^{-2}u_k =
\frac{x\cos\Phi_k}{(\lambda-\rho_k)^2}+\sin\Phi_k (A-\rho_k)^{-2}v_k
\end{equation}
From Equation~(\ref{newapprox}), the next eigenvector approximation $u_{k+1}$ in MSJD method satisfies the
following:
\begin{equation*}\label{equn46a}\relax
(A-\rho_k I)^2u_{k+1} = \frac{u_k}{\|(A-\rho_k I)^{-2} u_k\|}
\end{equation*}
Using Equation (\ref{equn48}), we have,
$$u_{k+1} =  \frac{\frac{x\cos\Phi_k}{(\lambda-\rho_k)^2}+\sin\Phi_k \cdot (A-\rho_k)^{-2}v_k}{\|(A-\rho_kI)^{-2}u_k\|}$$
Let $v_{k+1}$ be a unit vector such that $u_{k+1} = x\cos\Phi_{k+1} +v_{k+1}\sin\Phi_{k+1}$ and $v_{k+1}^\ast x = 0.$ Then 
$$\cos\Phi_{k+1} = \frac{\cos\Phi_k}{|\lambda-\rho_k|^2 \|(A-\rho_kI)^{-2}u_k\|}$$ 
$$\sin\Phi_{k+1} = \frac{\|(A-\rho_kI)^{-2}v_k\|\sin\Phi_k}{\|(A-\rho_kI)^{-2}u_k\|}$$ 
$$v_{k+1} = \frac{(A-\rho_kI)^{-2}v_k}{\|(A-\rho_kI)^{-2}v_k\|}$$ 
We have
\begin{align}\label{equn49}\relax
\tan\Phi_{k+1} =
\frac{\sin\Phi_k\|(A-\rho_kI)^{-2}v_k\|}{\cos\Phi_k 
(\lambda-\rho_k)^{-2}}
 = (\lambda-\rho_k)^{2}\|(A-\rho_kI)^{-2}v_k\|\tan\Phi_k
\end{align}
From equation~(\ref{equn47}), we have
$$ \rho_k = (x\cos\Phi_k +v_k\sin\Phi_k)^\ast A( x\cos\Phi_k +v_k\sin\Phi_k)$$
Since $Ax = \lambda x$ and $v_k^\ast x=0,$ it follows that $\rho_k = \lambda \cos^2\Phi_k+\rho(v_k)\sin^2\Phi_k$. Therefore
\begin{equation}\label{equn50}\relax
(\lambda-\rho_k) = (\lambda-\rho(v_k))\sin^2\Phi_k
\end{equation}
Using this in Equation~(\ref{equn49}), we obtain
\begin{equation}\label{equn51}\relax
\tan\Phi_{k+1} =
(\lambda-\rho(v_k))^2\|(A-\rho_kI)^{-2}v_k\|\sin^4\Phi_k\tan\Phi_k
\end{equation}
Since $v_k$ is orthogonal to $x,$ 
$$\|(A-\rho_kI)^{-2}v_k\| \leq \frac{1}{min_{\lambda_i \neq \lambda} |\lambda_i-\rho_k|^2}$$ 
Since $\Phi_k \rightarrow 0$,  Equation (\ref{equn50}) implies that $|\lambda-\rho_k| \rightarrow 0$ as $k \rightarrow \infty$. As $\rho_k$ converges to $\lambda$ as $k \rightarrow \infty$, there exists a real number $M$ such that 
\begin{equation}\label{equn52}\relax
|\lambda_i-\rho_k| \geq M~~~ \mbox{for large}~~ k
\end{equation}
for all $\lambda_i \neq \lambda$. Now Equations (\ref{equn51}) and (\ref{equn52})
together with $\tan\Phi = O(\Phi)$ and $\sin(\Phi) = O(\Phi)$ for small $\Phi$, imply that $|\Phi_{k+1}| \leq |\Phi_k|^5$ for large $k.$ 
\end{proof}
\section{Numerical experiments}

The MSJD method has been tested on many numerical examples for checking whether it really works, using Matlab R2014A on an Intel core 3 processor. Out of these we report three examples for demonstrating various features.  We compare the performance of the Jacobi-Davidson method and the proposed MSJD method by solving the corresponding correction equations
\begin{equation}\label{ojd}\relax
(I-uu^\ast )(A-\theta I)(I-uu^\ast )t = -(A-\theta I) u
\end{equation}
\begin{equation}\label{mjd}\relax
 (I-uu^\ast )(A-\theta I)^\ast (A-\theta I)(I-uu^\ast )t = -(I-uu^\ast )(A-\theta I)^\ast (A-\theta I)u
\end{equation} 
using Gaussian elimination as well as with the approximate solution obtained after a few steps of GMRES. Since the matrices in the left hand side of the correction equation~(\ref{ojd}) in Jacobi-Davidson method may become ill-conditioned, especially, when the matrix $A$ has multiple eigenvalues, it has been proposed in \cite{slei} to use the following equivalent form of (\ref{ojd}):
\begin{equation}\label{jds}\relax
\big((I-uu^\ast )A(I-uu^\ast )-\theta I \big )t = -(A-\theta I)u,~~t \perp u. 
\end{equation}
In a similar vein, we define the following correction equation in MSJD method: 
\begin{align}\label{mds}\relax
\begin{split}
\big((I-uu^\ast )A^\ast A(I-uu^\ast )- \theta (I-uu^\ast )A^\ast (I-uu^\ast ) ~~~~~~~~~~~~~~~~~~~~~~~~~~~~~~~\\ - \bar{\theta}(I-uu^\ast )A(I-uu^\ast )+|\theta |^2 I\big)t  = -(I-uu^\ast )(A-\theta I)^\ast (A-\theta I)u, t \perp u
\end{split}
\end{align}
which is theoretically equivalent to (\ref{mjd}).\\\\
In the following examples, we compare the performance of the correction equation~(\ref{ojd}) in Jacobi-Davidson method with the correction equation~(\ref{mjd}) in Modified Jacobi-Davidson method. We also compare the numerical results obtained using the correction equation~(\ref{jds}) in Jacobi-Davidson method with the correction equation~(\ref{mds}) in Modified Jacobi-Davidson method. We also demonstrate in the following examples that the performance of JD method is numerically different for the two theoretically equivalent  correction equations (\ref{ojd}) and (\ref{jds}). This is because of the presence of rounding errors in floating point arithmetic. In a similar vein, we also demonstrate that the MJD performs numerically different for the two theoretically equivalent correction equations (\ref{mjd}) and (\ref{mds}). \\

In all the figures,``JD" means either the correction equation~(\ref{ojd}) or (\ref{jds}) in Jacobi-Davidson method is used whereas ``MJD" means either the correction equation~(\ref{mjd}) or (\ref{mds}) is used. In order to check the performance of the proposed method, without restarting, we consider the following algorithm:
\begin{alg}\label{alg1}\title{Unrestarted algorithm}\relax 
\textsf{1. Solve the correction equation, either (\ref{ojd}) or  (\ref{jds}) in Jacobi-Davidson method or either (\ref{mjd}) or (\ref{mds}) in Modified Jacobi Davidson method.\\
2. Expand the subspace with the vector obtained in Step 1, as explained in Section~2.\\
3. Apply Rayleigh-Ritz projection/Harmonic Rayleigh-Ritz projection for eigenvalues, and calculate the norm of residuals associated with refined Ritz/Harmonic Ritz vectors. If they reach $`tol',$ stop. Otherwise, go to Step 1.}
\end{alg} 
When a good approximation to eigenvalues in the interior of the spectrum is required we use Harmonic Rayleigh-Ritz projection \cite{int} in Algorithm~1.  The results on performance of Algorithm~1 for three Examples are reported below.

\begin{example}\label{eg3}\relax
Consider the diagonal matrix $A$ of order $100$ with diagonal elements as $(\frac{j}{100})^2-0.8$ for $j = 1,2...100$. Take the initial vector as the vector with each entry equal to $1$.  It is required to compute an eigenvalue with the smallest absolute value.
\end{example}
This matrix is same as in Example 3 in \cite{slei} and is used to show that the proposed method can also be used in computing interior eigenvalues. As the method is not restarted, in this and next example, iteration number denotes the size of subspace used in  Harmonic Projection. For this example, we used Harmonic projection in Algorithm~\ref{alg1} and Harmonic Ritz vectors in the correction equations (\ref{ojd}) and (\ref{mjd}).
\begin{figure}[!htb]
\begin{minipage}{0.4975\linewidth}
\begin{center} 
\includegraphics[width = 2in,height=2in]{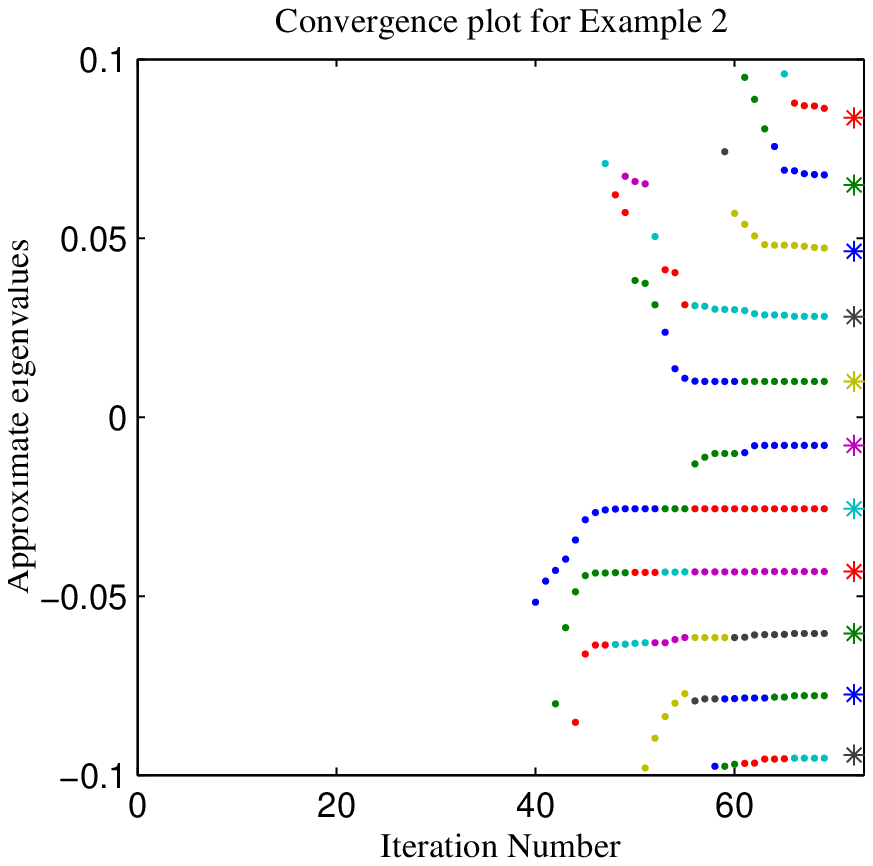}
\caption{\scriptsize \sl \hspace{4pt}Using GMRES
 for equation (\ref{ojd}) with \\ Harmonic projection} 
\label{fig:cegrit20} 
\end{center} 
\end{minipage}
\mbox{\hspace{0.5cm}}
\begin{minipage}{0.4975\linewidth}
\begin{center} 
\includegraphics[width = 2in,height=2in]{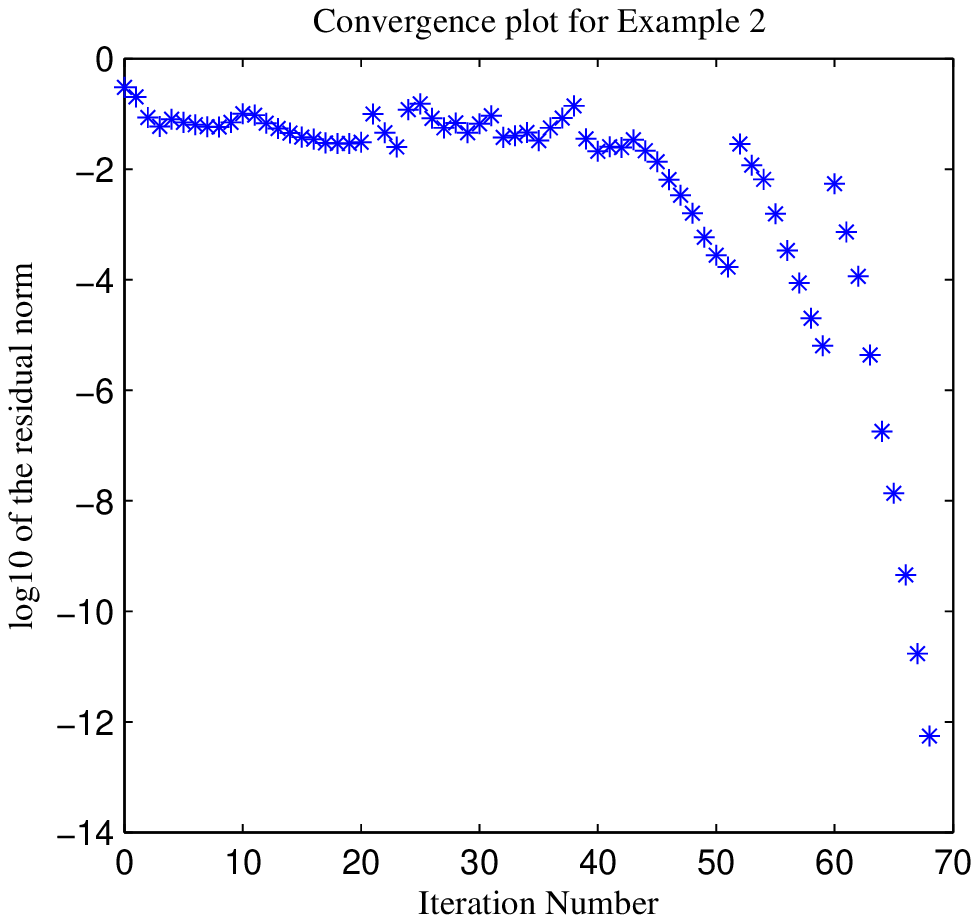}
\caption{\scriptsize \sl  Using GMRES
 for equation (\ref{ojd}) with \\ Harmonic projection} 
\label{fig:cegrit19} 
\end{center} 
\end{minipage}
\end{figure}
\begin{figure}[!htb]
\begin{minipage}{0.4975\linewidth}
\begin{center} 
\includegraphics[width = 2in,height=2in]{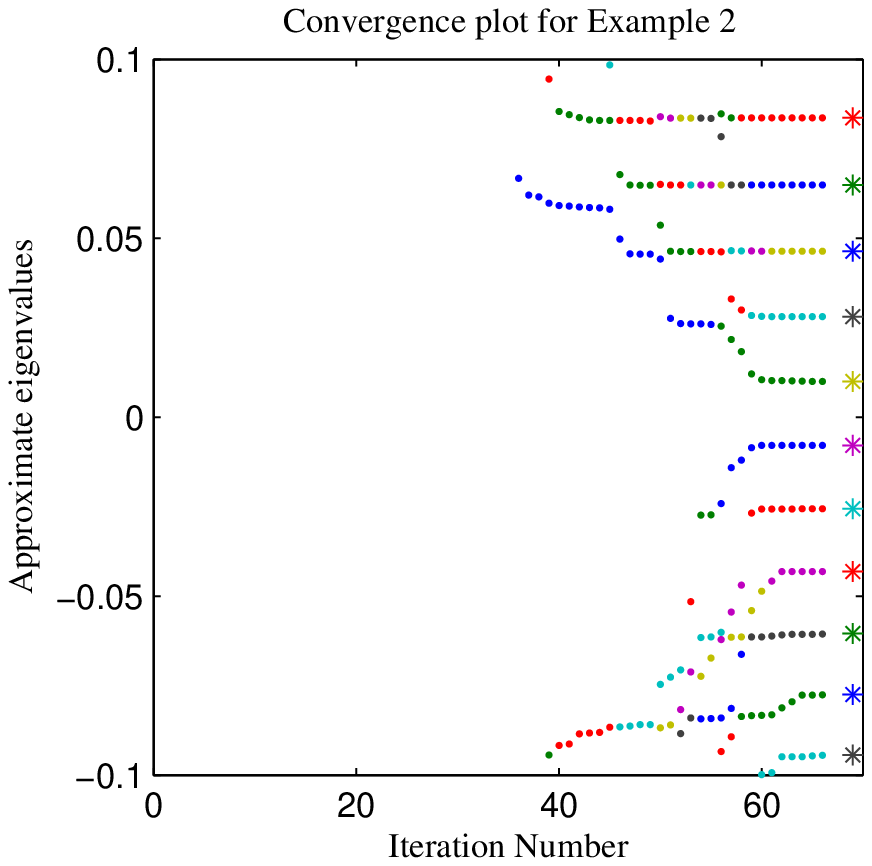}
\caption{\scriptsize \sl Using GMRES
 for equation (\ref{mjd}) with \\Harmonic projection} 
\label{fig:cegrit22} 
\end{center} 
\end{minipage}
\mbox{\hspace{0.5cm}}
\begin{minipage}{0.4975\linewidth}
\begin{center} 
\includegraphics[width = 1.9in,height=1.9in]{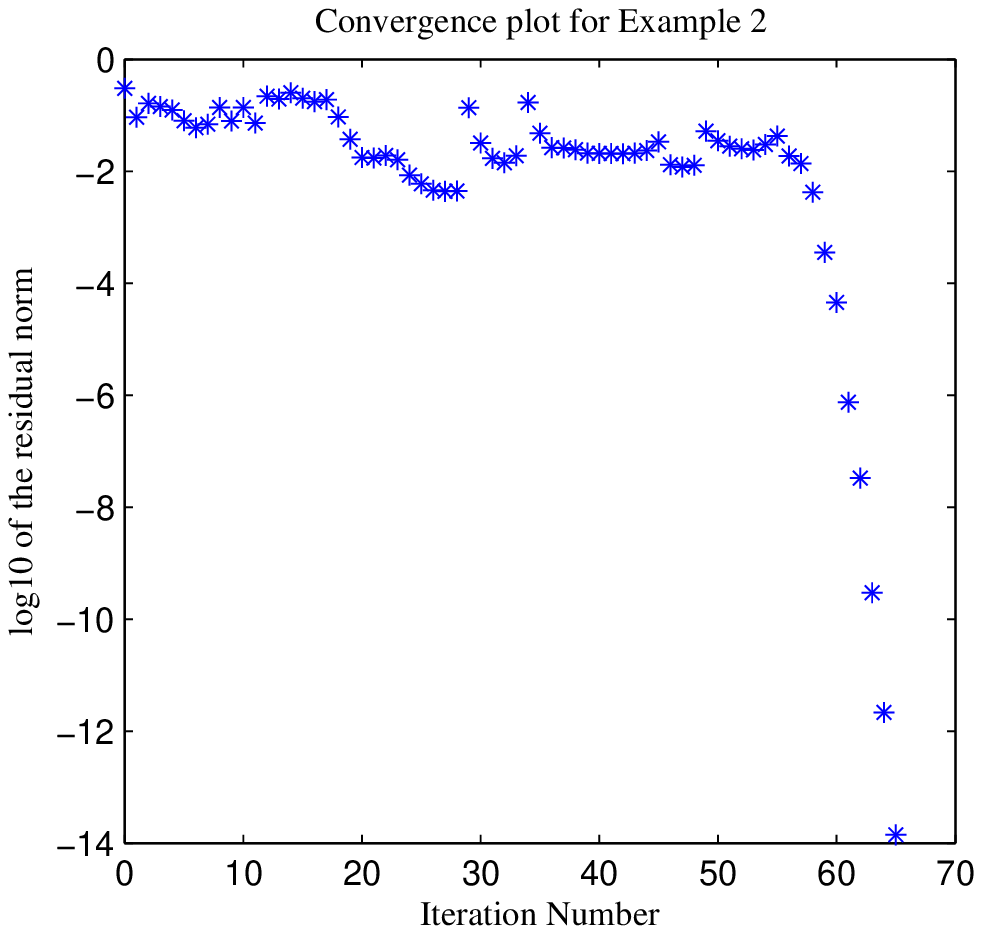}
\caption{\scriptsize \sl  Using GMRES for equation (\ref{mjd}) with \\Harmonic projection} 
\label{fig:cegrit21} 
\end{center} 
\end{minipage}
\end{figure}

 With an approximate solution of Equations (\ref{ojd}) and (\ref{mjd}) obtained using $8$ steps of GMRES, the convergence to the required eigenvalue is observed at  iterations $69 $ and $66,$ respectively. The $\log_{10}$ of the norm of the residuals are $-12.26$ and $-13.65,$ respectively. The numerical results with the correction equation (\ref{ojd}) are shown in Figures \ref{fig:cegrit20}-\ref{fig:cegrit19}. Figures \ref{fig:cegrit22}-\ref{fig:cegrit21} show the results, when the correction equation (\ref{mjd}) is used. In Figures \ref{fig:cegrit20} and \ref{fig:cegrit22}, $`\ast'$ represents the exact eigenvalues of $A$. From these Figures, it is clear that convergence results of approximate eigenvalues in both the methods are on par.

\begin{figure}[!htb]
\begin{minipage}{0.4975\linewidth}
\begin{center} 
\includegraphics[width = 2in,height=2in]{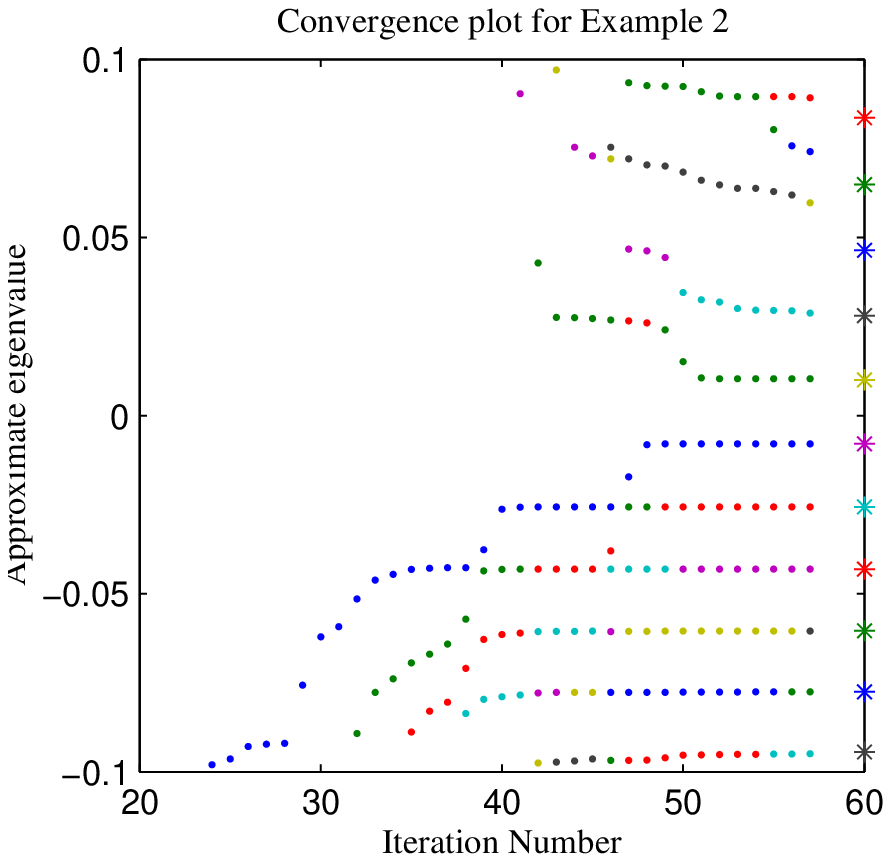}
\caption{\scriptsize \sl Using GMRES  for equation (\ref{jds}) with \\Harmonic projection} 
\label{fig:cegrit24} 
\end{center} 
\end{minipage}
\mbox{\hspace{0.5cm}}
\begin{minipage}{0.4975\linewidth}
\begin{center} 
\includegraphics[width = 2in,height=2in]{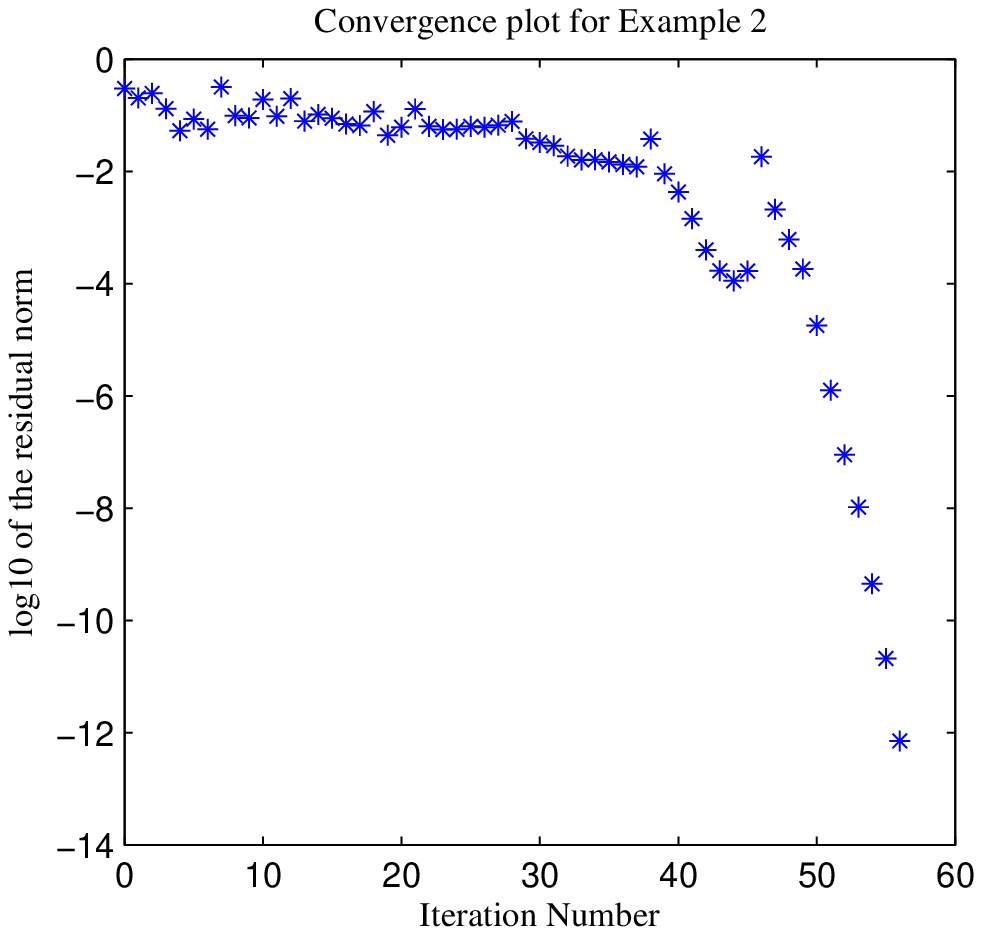}
\caption{\scriptsize \sl  Using GMRES for equation (\ref{jds}) with \\Harmonic projection} 
\label{fig:cegrit23} 
\end{center} 
\end{minipage}
\end{figure}

\begin{figure}[htb]
\begin{minipage}{0.4975\linewidth}
\begin{center} 
\includegraphics[width = 2in,height=2in]{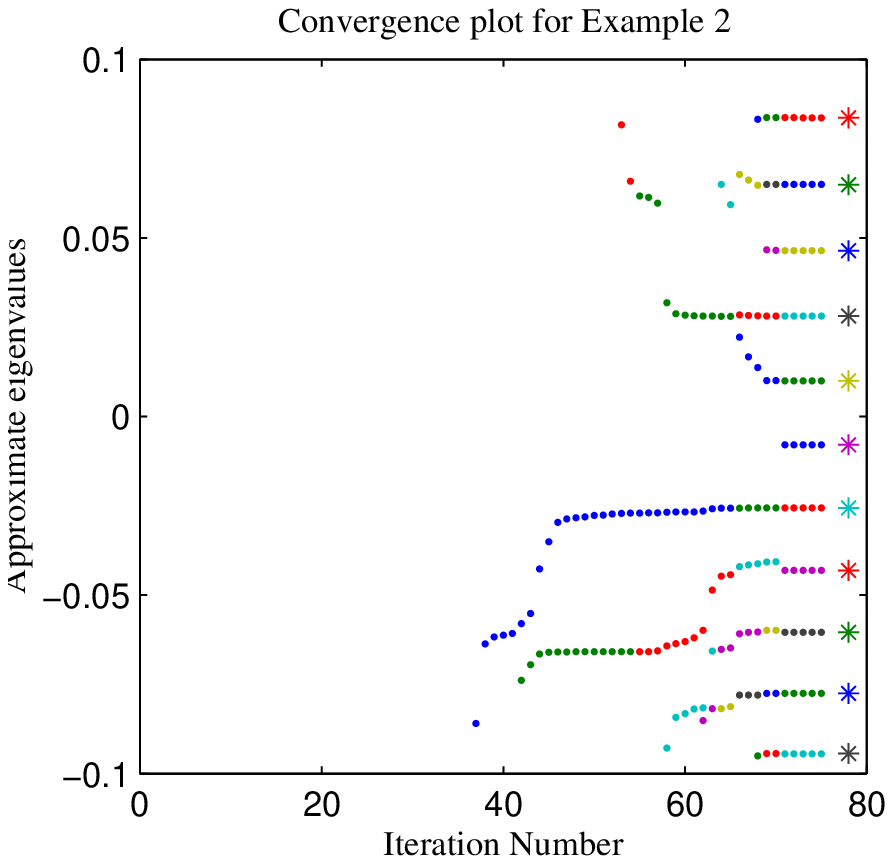}
\caption{\scriptsize \sl Using GMRES for equation (\ref{mds}) with \\Harmonic projection} 
\label{fig:cegrit26} 
\end{center} 
\end{minipage}
\mbox{\hspace{0.5cm}}
\begin{minipage}{0.4975\linewidth}
\begin{center} 
\includegraphics[width = 2in,height=2in]{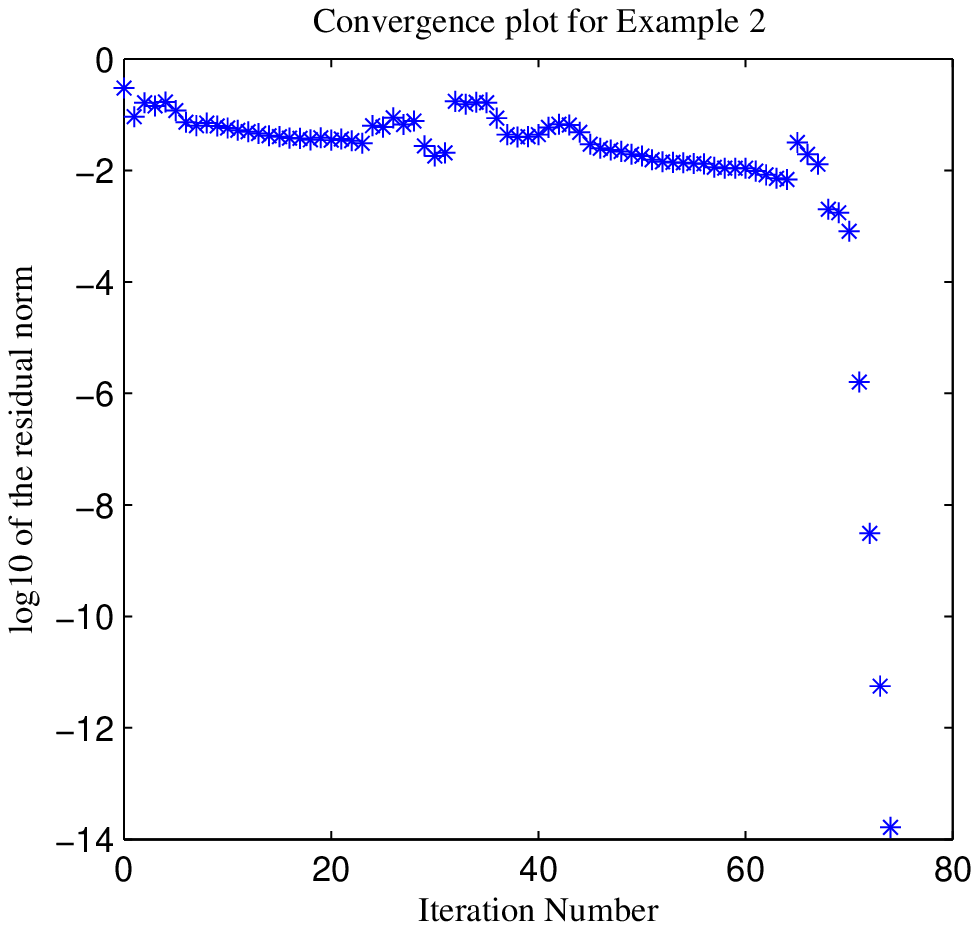}
\caption{\scriptsize \sl  Using GMRES for equation (\ref{mds}) with \\Harmonic projection} 
\label{fig:cegrit25} 
\end{center} 
\end{minipage}
\end{figure}

Similarly, using the solutions of the correction equations (\ref{jds}) and (\ref{mds}) obtained using $8$  steps of GMRES, convergence occurs at iterations $57$ and $75,$ with $\log_{10}$ of residual norms as $-12.15$ and $-13.78,$ respectively. Convergence for these cases are shown in Figures \ref{fig:cegrit24}-\ref{fig:cegrit23} and \ref{fig:cegrit26}-\ref{fig:cegrit25}, respectively. Using the correction equation (\ref{jds}), convergence of norms of residual vectors is faster when compared to MJD using the correction equation (\ref{mds}).

If Gaussian elimination is used instead of GMRES to solve the correction equations  (\ref{jds}) and (\ref{mds}), with the correction equation (\ref{jds}), convergence to the required eigenvalue occurred at $57^{th}$ iteration, whereas with the correction equation (\ref{mds}),  convergence is achieved at $46^{th}$ iteration. Thus, when the correction equations (\ref{jds}) and (\ref{mds}) are solved by using the Gaussian elimination, we noticed that with the  correction equation (\ref{mds}) the convergence to  the required eigenvalue is faster than that with the correction equation (\ref{jds}).  This is exactly opposite to the scenario that we observed when the correction equations (\ref{jds}) and (\ref{mds}) are solved approximately by using the GMRES.

In Examples 2, even though the matrix is symmetric, we used the GMRES to solve the correction equations approximately as we did not take the advantage of this for generating  matrices $H_k$ in JD and MJD methods (For details, see Section-2). In the following example, we consider a non-Hermitian matrix. We observe  that like Jacobi-Davidson method, the modified method is also useful to approximate non-real eigenvalues. 

\begin{example}\label{eg4}\relax
  The matrix in this Example is a block diagonal matrix $diag\{A1,A2\}$ where
$$ A1 =\begin{bmatrix}
 0.8+0.1i & 0\\ 0 & 0.8-0.1i
\end{bmatrix} $$
and $A2$ is the matrix of Example 2. Again, each entry of the initial vector is taken as $1$. We apply Harmonic projection with shift $0.81+0.08i$ to find the eigenvalue $0.8+0.1i$.
\end{example}

\begin{figure}[!htb]
\begin{minipage}{0.4975\linewidth}
\begin{center} 
\includegraphics[width = 2in,height=2in]{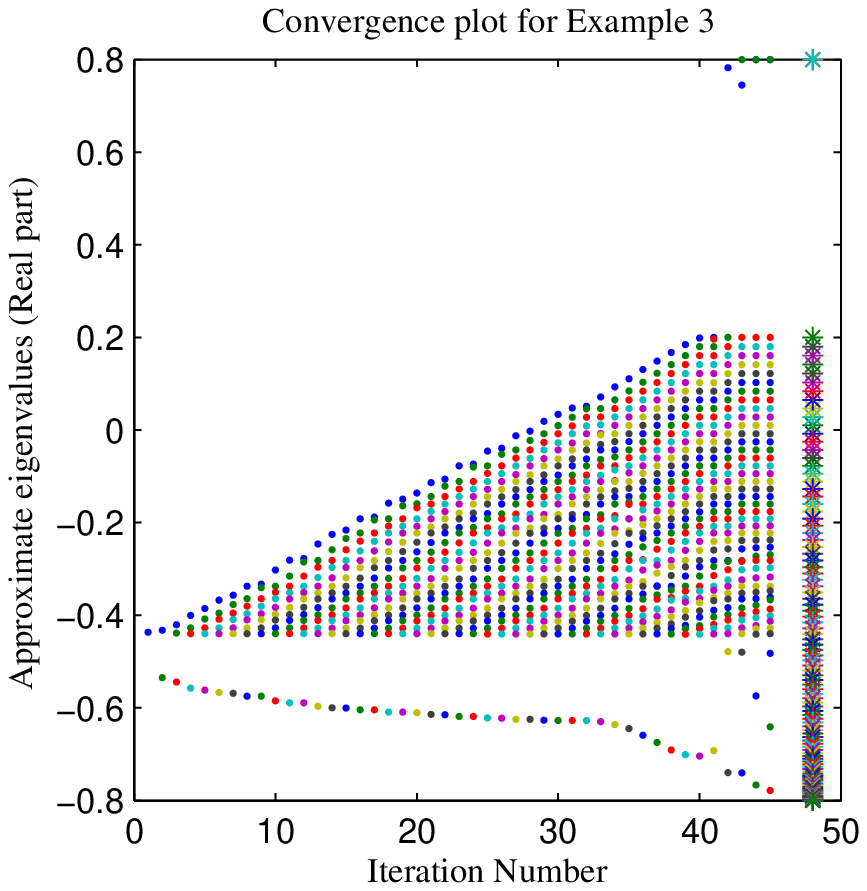}
\caption{\scriptsize \sl Using GMRES for equation (\ref{jds}) \\with Harmonic projection} 
\label{fig:cegrit32} 
\end{center} 
\end{minipage}
\mbox{\hspace{0.5cm}}
\begin{minipage}{0.4975\linewidth}
\begin{center} 
\includegraphics[width = 2in,height=2in]{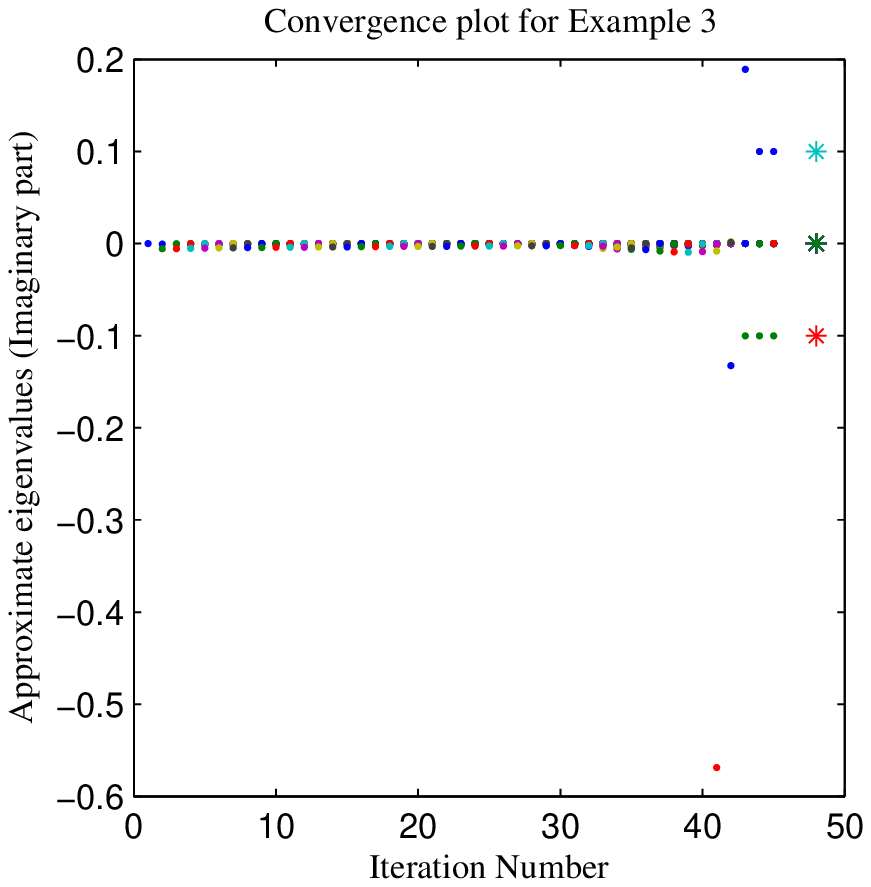}
\caption{\scriptsize \sl  Using GMRES for equation (\ref{jds}) \\with Harmonic projection} 
\label{fig:cegrit33} 
\end{center} 
\end{minipage}
\end{figure}

\begin{figure}[!htb]
\begin{minipage}{0.4975\linewidth}
\begin{center}
\includegraphics[width = 2in,height=2in]{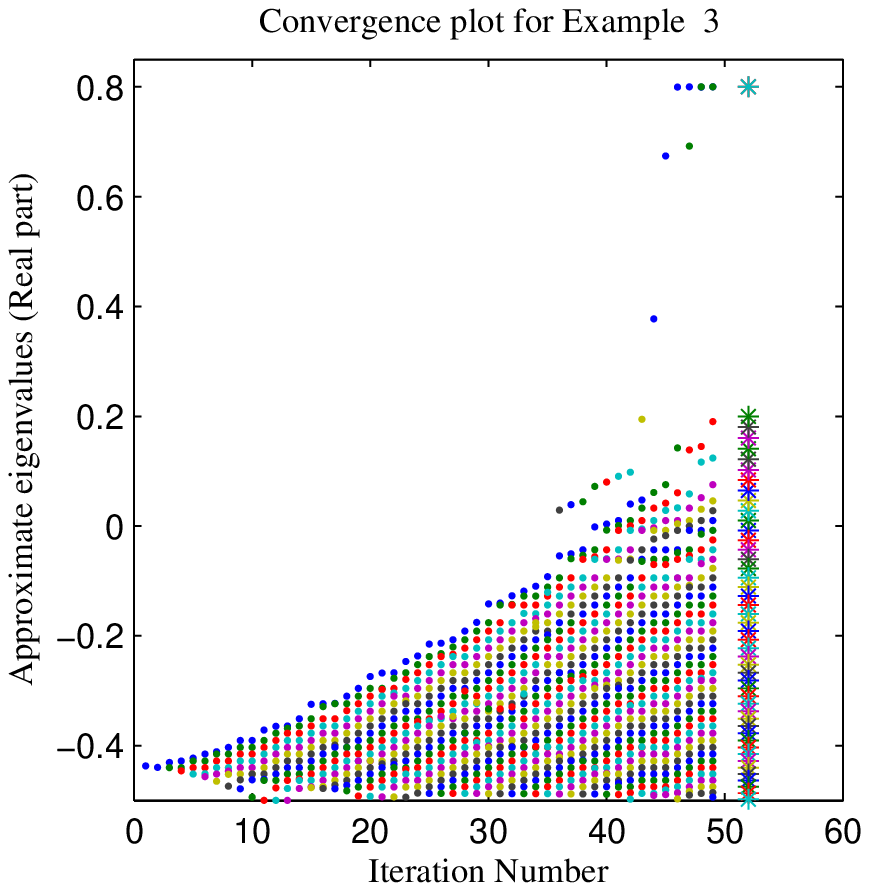}
\caption{\scriptsize \sl Using GMRES for equation (\ref{mds}) \\with Harmonic projection} 
\label{fig:cegrit28} 
\end{center}
\end{minipage}
\mbox{\hspace{0.5cm}}
\begin{minipage}{0.4975\linewidth}
\begin{center} 
\includegraphics[width = 2in,height=2in]{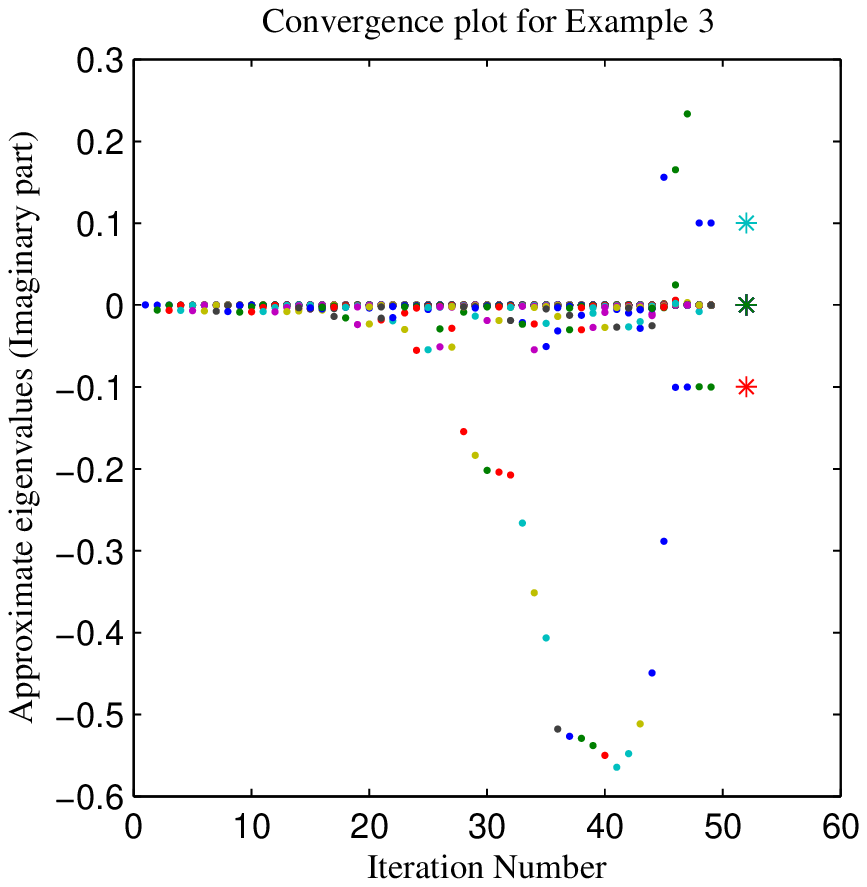}
\caption{\scriptsize \sl  Using GMRES for equation (\ref{mds}) \\with Harmonic projection} 
\label{fig:cegrit29} 
\end{center} 
\end{minipage}
\end{figure}

Here, we compare the results obtained using the correction equations (\ref{jds}) and (\ref{mds}). In both the approaches the resulting linear systems are solved using Gaussian elimination. With these two correction equations the convergence to the desired eigenvalue  occurred at $44^{th}$ and $48^{th}$ iterations, respectively. In both the cases, apart from the desired eigenvalue,the algorithm also finds the eigenvalue $0.8-0.1i$ which is far from the shift.  The same is observed, using Gaussian elimination for solving the correction equation (\ref{ojd}). 
For the correction equation (\ref{jds}), Figures \ref{fig:cegrit32} and \ref{fig:cegrit33} show the convergence history of the real parts and imaginary parts of the harmonic Ritz values, respectively.  Similarly, for the correction equation (\ref{mds}), the Figures \ref{fig:cegrit28} and \ref{fig:cegrit29} show the convergence history of the harmonic Ritz values. In all these figures, we have used the same symbols as in the previous example.
 \begin{figure}[!htb]
\begin{minipage}{0.4975\linewidth}
\begin{center} 
\includegraphics[width = 2in,height=2in]{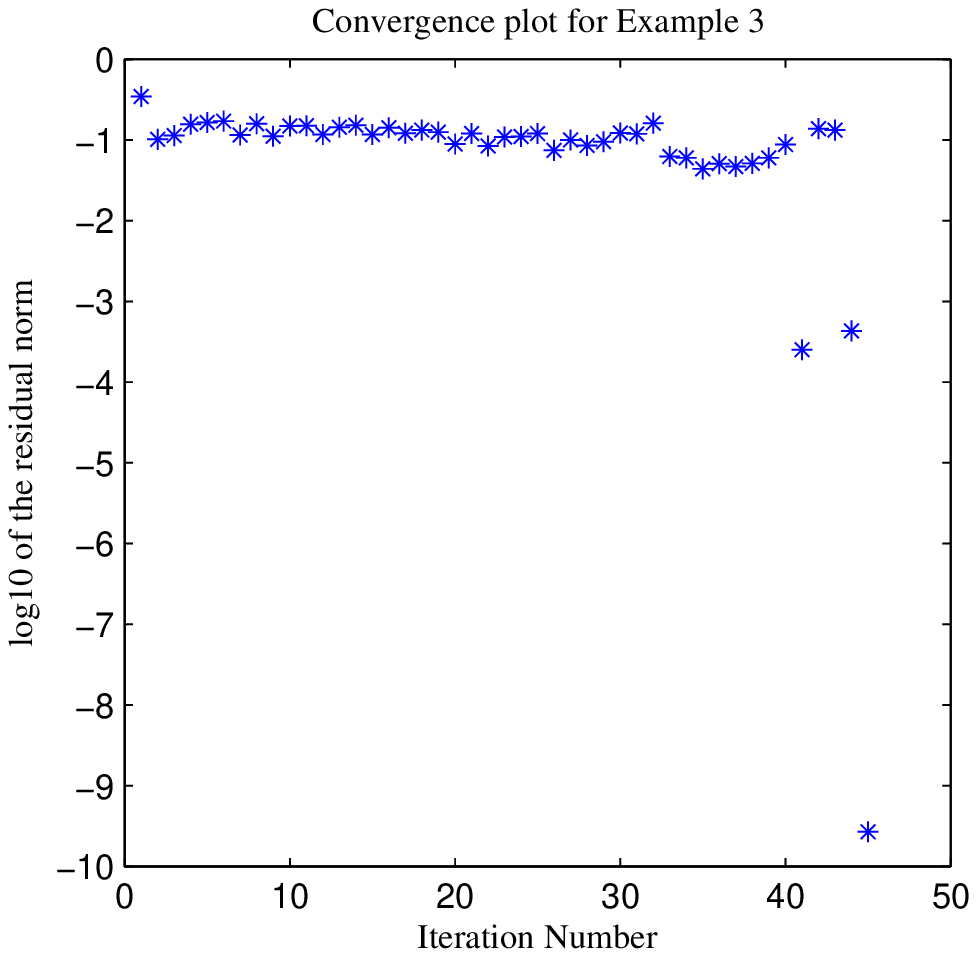}
\caption{\scriptsize \sl  Using GMRES for equation (\ref{jds}) \\with Harmonic projection} 
\label{fig:cegrit31} 
\end{center} 
\end{minipage}
\mbox{\hspace{0.5cm}}
\begin{minipage}{0.4975\linewidth}
\begin{center} 
\includegraphics[width = 2in,height=2in]{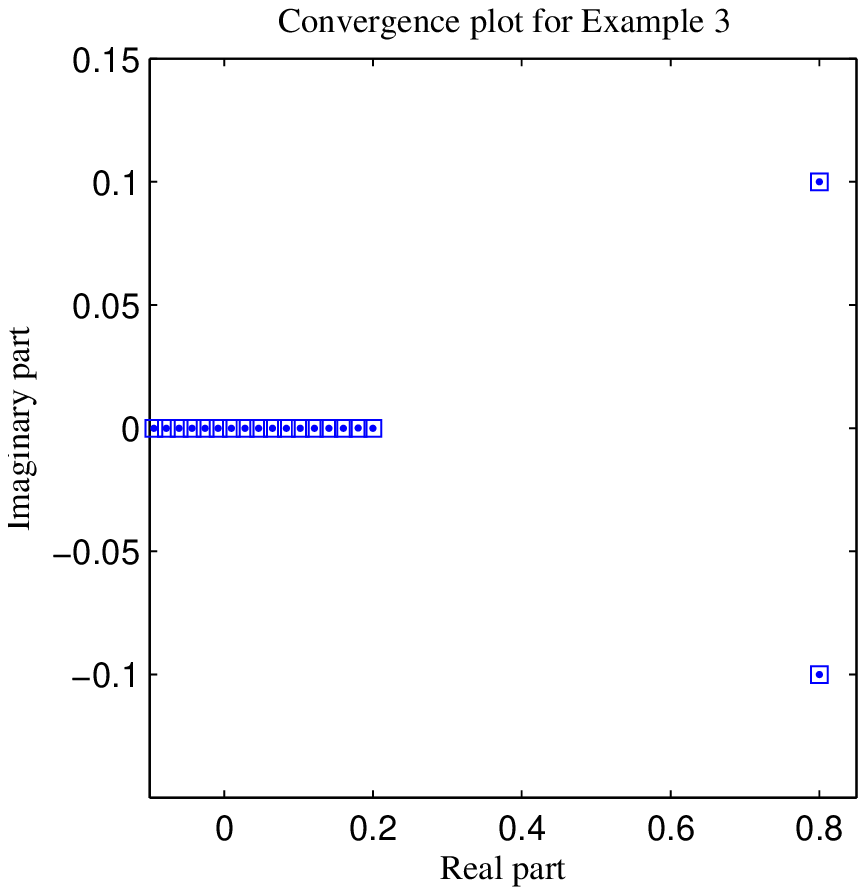}
\caption{\scriptsize \sl Using GMRES for equation (\ref{jds}) \\with Harmonic projection} 
\label{fig:cegrit34} 
\end{center} 
\end{minipage}
\end{figure}

\begin{figure}[!htb]
\begin{minipage}{0.4975\linewidth}
\begin{center} 
\includegraphics[width = 2in,height=2in]{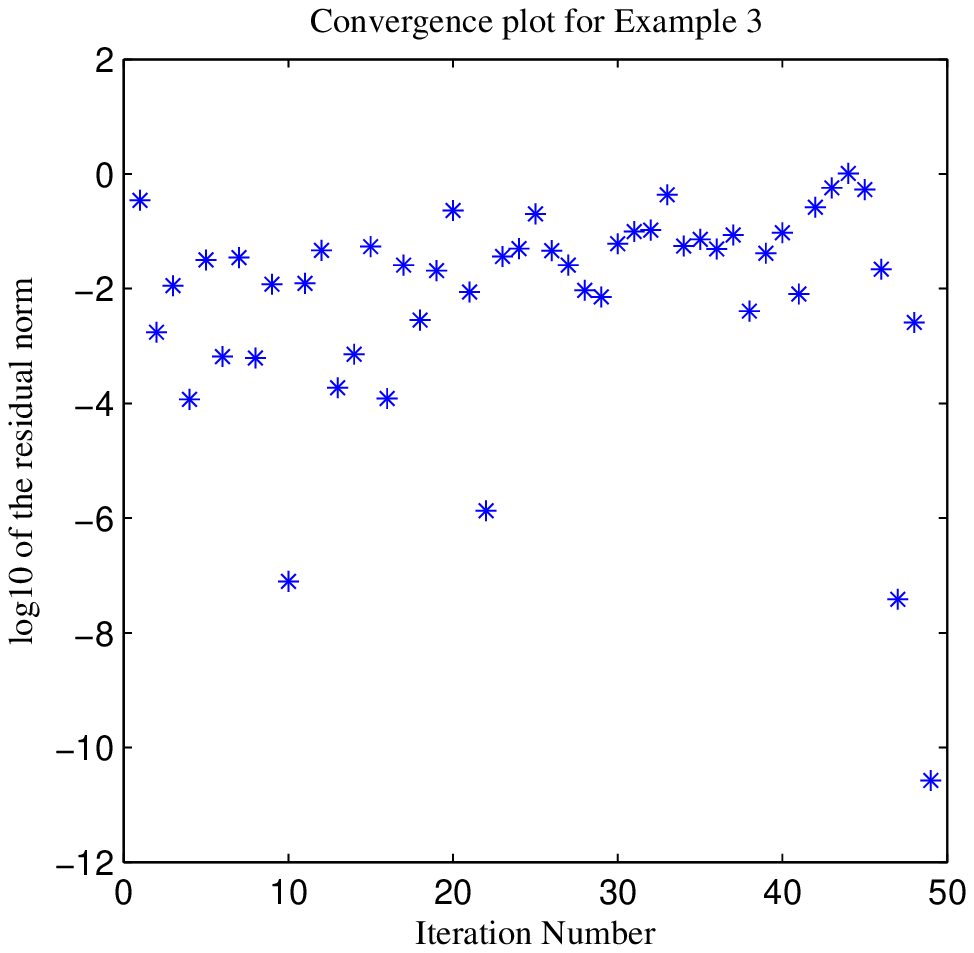}
\caption{\scriptsize \sl  Using GMRES for equation (\ref{mds}) \\with Harmonic projection} 
\label{fig:cegrit27} 
\end{center} 
\end{minipage}
\mbox{\hspace{0.5cm}}
\begin{minipage}{0.4975\linewidth}
\begin{center} 
\includegraphics[width = 2in,height=2in]{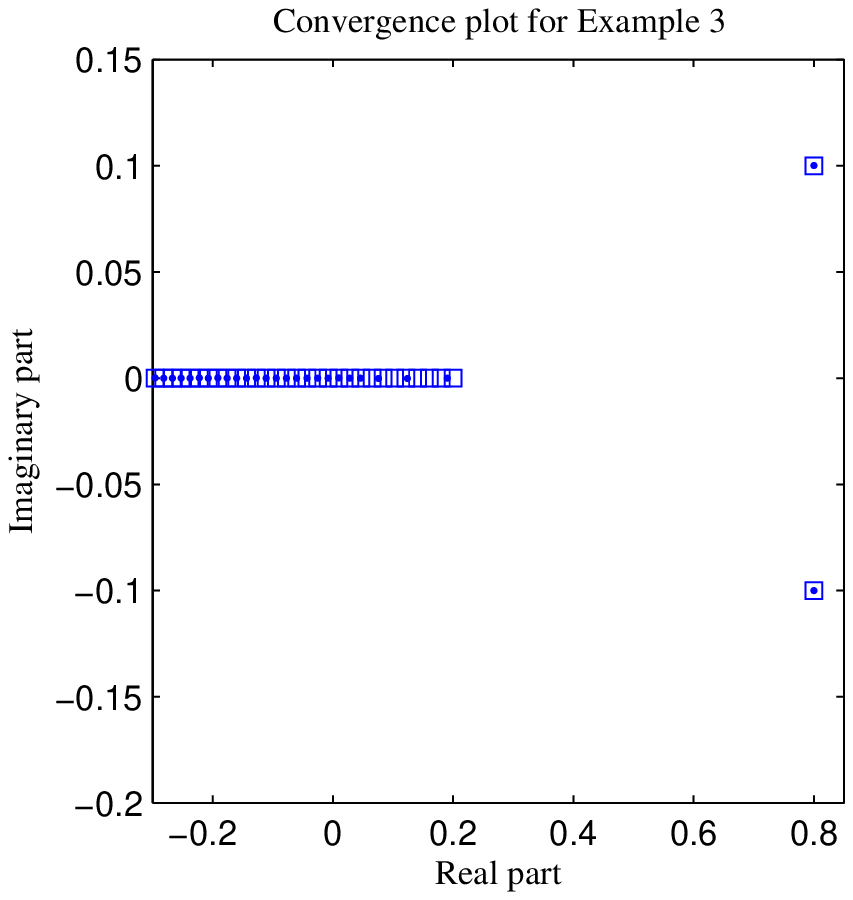}
\caption{\scriptsize \sl Using GMRES for equation (\ref{mds}) \\with Harmonic projection} 
\label{fig:cegrit30} 
\end{center} 
\end{minipage}
\end{figure}

For the correction equations (\ref{jds}) and (\ref{mds}), the convergence of norm of residual vectors are shown in the Figures \ref{fig:cegrit31} and \ref{fig:cegrit27}, respectively.For the correction equations (\ref{jds}) and (\ref{mds}), the harmonic Ritz values at the final iteration where the convergence occurred are shown in Figures \ref{fig:cegrit34} and \ref{fig:cegrit30}, respectively. In these Figures, we marked the exact eigenvalues of $A$ with squares. Dots represent the Harmonic Ritz values obtained at the iteration where convergence occurs. It is easy to observe from these two figures that the accurate approximation to larger number of eigenvalues are obtained using the correction equation (\ref{mds})when compared to the correction equation (\ref{jds}). 
 
In case of solving the correction equation (\ref{mjd}) using the Gaussian elimination, the eigenvalue approximations converges to an eigenvalue other than the desired one. When the correction equations (\ref{ojd}), (\ref{jds}) and (\ref{mjd}), (\ref{mds}) are solved approximately using $10$ steps of GMRES, the approximation converges to the eigenvalue $2.000+0.000i$, which is not the desired eigenvalue.

The next example shows that the new method also works for large sparse matrices.
\begin{example}\label{eg5}\relax
We consider the matrix $A$ as SHERMAN4, a sparse matrix of order $1104$, taken from Harwell-Boeing set of test matrices. The smallest eigenvalue $0.030726$ (accurate upto $5$ decimal places) is required. MATLAB command `eigs' produces the result as $ 3.072570776499973e-02$. All entries in the initial vector are equal to $1$. Rayleigh-Ritz projection and Refined Ritz vectors are used for approximating the eigen pairs.
\end{example}
\begin{table}[!htb]\label{table2}\relax
\caption{Comparison of JD and MJD methods using either Gaussian elimination solution or approximate solution of correction equations}
\begin{adjustbox}{width = 1\textwidth}
\begin{scriptsize}
\begin{tabular}{|c|c|c|c|c|}
\hline 
Equn &  method of solving  & iteration  & Ritz value & Norm of residual vector   \\ 
No.~~ & Linear system & Number & &\\
\hline 
\ref{ojd} & Gaussian elimination  &  10 & 3.072570776430865e-02 & 8.937205079499508e-11 \\ 
\hline 
\ref{mjd} & Gaussian elimination  &  11 & 3.072570776499898e-02 & 2.682680808082383e-14\\ 
\hline 
\ref{jds} & Gaussian elimination  &  5 &  3.072570776525444e-02 & 1.169743153032539e-12\\ 
\hline 
\ref{mds} & Gaussian elimination  &  11 & 3.072570776499969e-02 & 1.881587896753183e-14 \\ 
\hline 
\end{tabular}
\end{scriptsize}
\end{adjustbox}
\end{table}
In Table~1, we give  the numerical results of Jacobi-Davidson and MSJD method for the matrix Sherman4. Fast convergence is observed when Gaussian elimination is used for solving the correction equation (\ref{jds}) compared to using the other correction equations.  The comparison of convergence  for the correction equations (\ref{jds}) and (\ref{mds})  is done in Figures \ref{fig:cegrit35}-\ref{fig:cegrit36}.
\begin{figure}[!htb]
\begin{minipage}{0.4975\linewidth}
\begin{center} 
\includegraphics[width = 2in,height=2in]{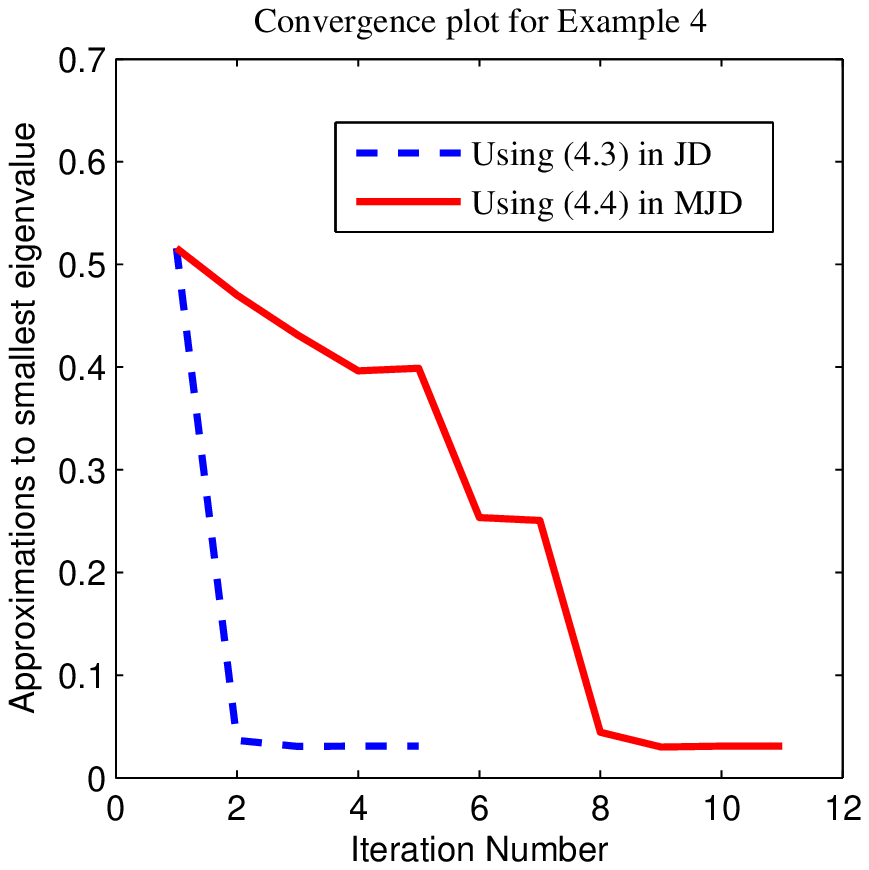}
\caption{\scriptsize \sl  Convergence of Ritz values with\\ correction equations (\ref{jds}) and (\ref{mds})} 
\label{fig:cegrit35} 
\end{center} 
\end{minipage}
\mbox{\hspace{0.5cm}}
\begin{minipage}{0.4975\linewidth}
\begin{center} 
\includegraphics[width = 2in,height=2in]{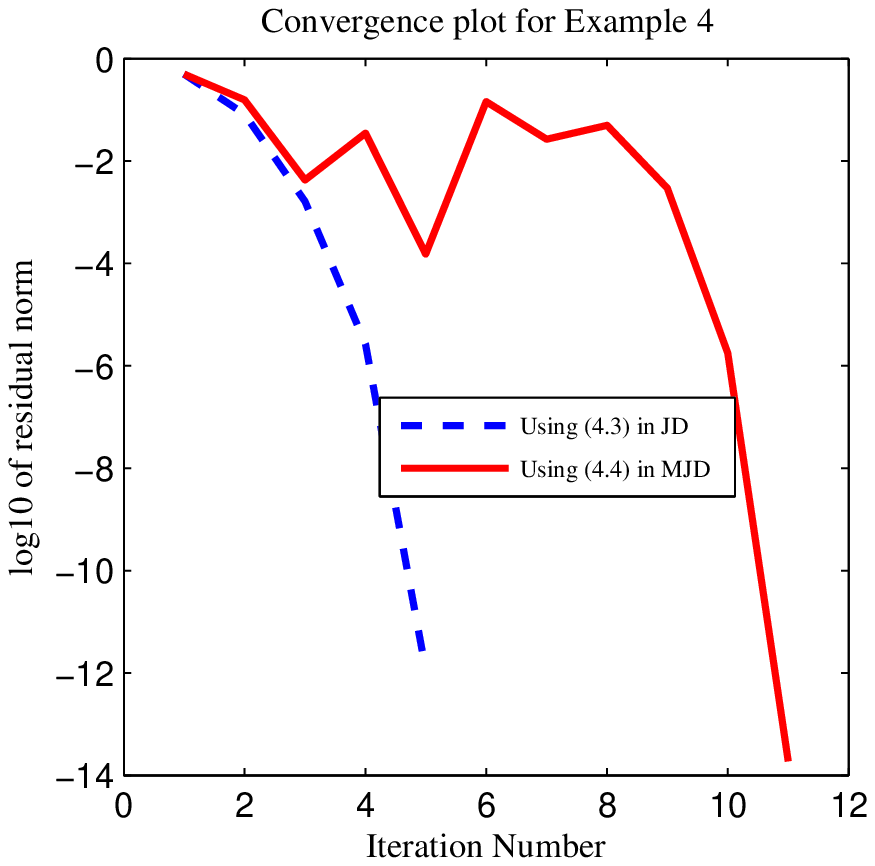}
\caption{\scriptsize \sl  Convergence of Residual norms with\\correction equations (\ref{jds}) and (\ref{mds})} 
\label{fig:cegrit36} 
\end{center} 
\end{minipage}
\end{figure}

\begin{figure}[!htb]
\begin{minipage}{0.4975\linewidth}
\begin{center} 
\includegraphics[width = 2in,height=2in]{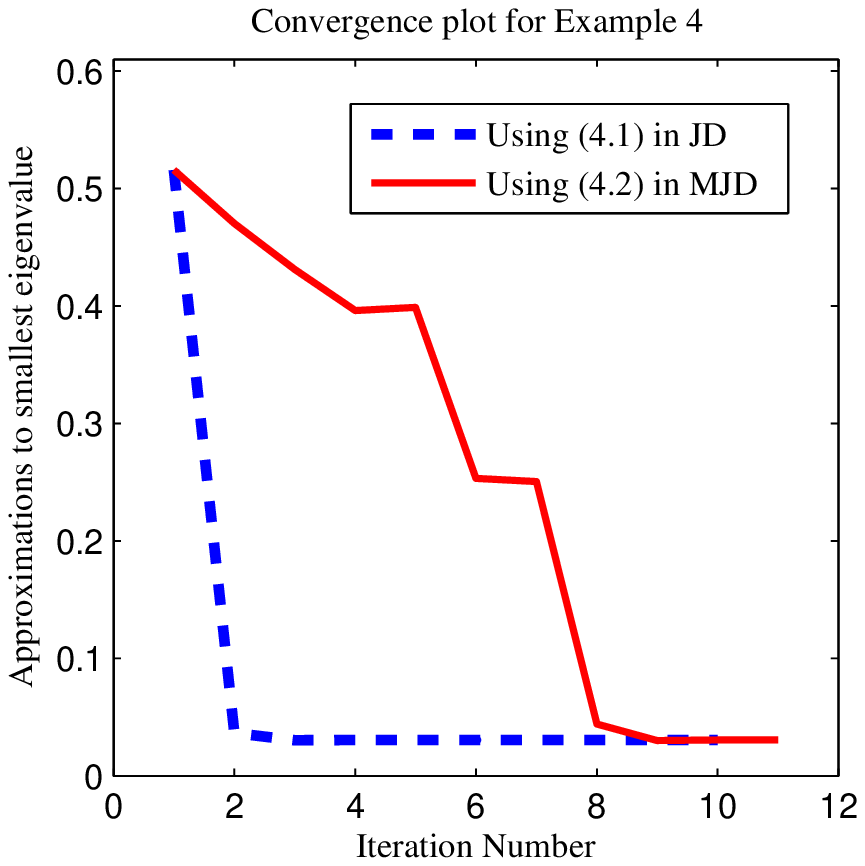}
\caption{\scriptsize \sl  Convergence of Ritz values with\\ correction equations (\ref{ojd}) and (\ref{mjd})} 
\label{fig:cegrit37} 
\end{center} 
\end{minipage}
\mbox{\hspace{0.5cm}}
\begin{minipage}{0.4975\linewidth}
\begin{center} 
\includegraphics[width = 2in,height=2in]{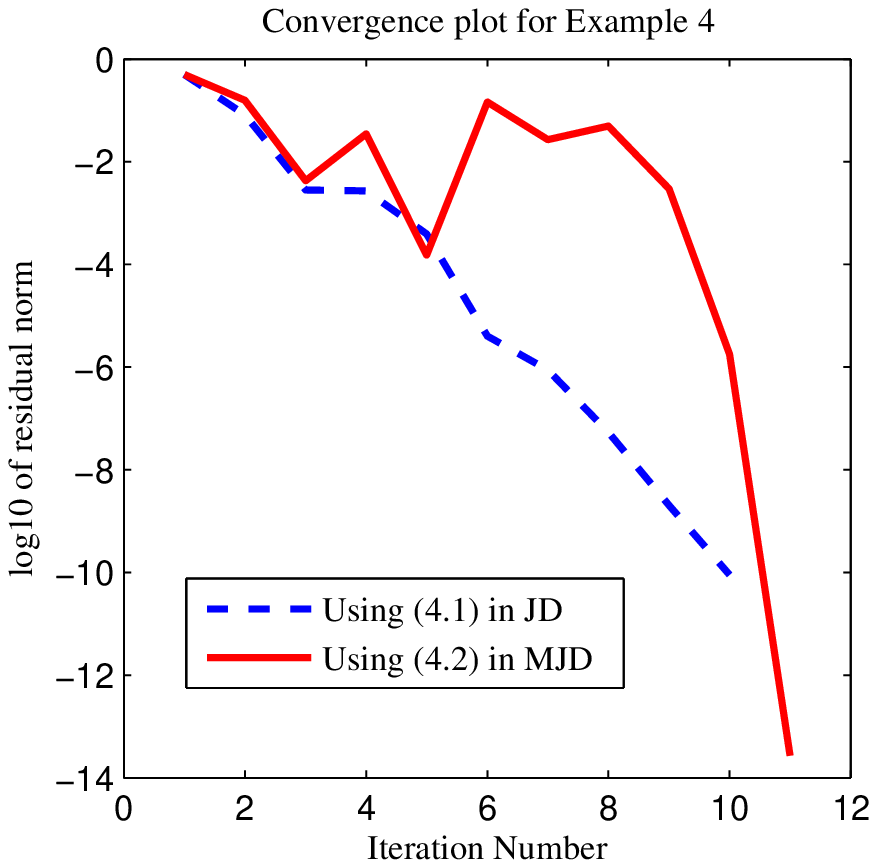}
\caption{\scriptsize \sl  Convergence of Residual norms with\\ correction equations (\ref{ojd}) and (\ref{mjd})} 
\label{fig:cegrit38} 
\end{center} 
\end{minipage}
\end{figure}

With Gaussian elimination for solving the correction equation (\ref{mjd}), the convergence occurs at the iteration number $11,$ whereas with the correction equation (\ref{ojd}), the convergence is reached at the iteration number $10.$ In both the approaches, the obtained eigenvalue approximation is  accurate upto $10$ decimal places. The Comparison of convergence of eigenvalue approximations and norms of residual vectors in these two cases is done in Figures \ref{fig:cegrit37} and \ref{fig:cegrit38}, respectively. 
We observe that with Gaussian elimination for solving the correction equations in the Jacobi-Davidson and the new method, the results are on par

\section{With restarting}
It is well known that for symmetric matrices, Rayleigh-Ritz projection over large subspaces may give good approximations to an eigen pair. But as the size of a subspace increases, the cost associated with computing an eigen pair  also increases. Further, if the size of the given matrix is very large, the space complexity in the computation may become  practically unmanageable. For this reason, the method with restarting is favourable. In the following examples, we check the performance of the MSJD method with restarting. 
\begin{example}\label{eg1}\relax
Consider the matrix $A$ as the first Example in \cite{slei}, which is a diagonally dominant tridiagonal matrix of order $200$ with diagonal elements $a_{i,i} = 2.4+i/2$  for $ i < 200$, $a_{200,200} = 2.4+200/1.5,$ and with each entry on the super-diagonal and sub-diagonal as $1$. We take the initial vector $v1 = (0.03,0.03,....0.03,1)^{\ast}$ as in \cite{slei}. The corresponding Rayleigh quotient with respect to $A$ is $1.632770531196111e+02$. Our goal is to approximate the largest eigenvalue. Such a preliminary approximation is obtained by using Matlab command $\textit{`eig'}$, which computes it as $2.561474561181774e+02$.%(5,2) case.  
\end{example}

We compare the performance of the proposed method with the Jacobi-Davidson method by restarting the algorithm after the size of subspace becomes $3$.  Table~2 shows a summary of numerical results. 

\begin{table}[!htb]\label{t1}\relax
\caption{Using Gaussian elimination method to solve correction equation for the matrix in Example 5}
\begin{adjustbox}{width = 1\textwidth}
\begin{scriptsize}
\begin{tabular}{|c|c|c|c|c|}
\hline 
Equn &  method of solving  & Restart  & Ritz value & Norm of residual vector   \\ 
No.~~ &  Linear system  &Number & &\\
\hline 
\ref{ojd} &  Gaussian elimination  &  2  & 2.561474561181777e+02 &  7.503060878161262e-12\\ 
\hline 
\ref{mjd} &  Gaussian elimination  &  2  & 2.561474561181780e+02 & 6.311243610822153e-13\\ 
\hline
\ref{jds} &  Gaussian elimination  &  2 &  2.561474561182365e+02 &5.862284941673252e-11\\ 
\hline  
\ref{mds} &  Gaussian elimination  &  2 &  2.561474561181781e+02 & 3.095655387252406e-13 \\ 
\hline
\end{tabular}
\end{scriptsize}
\end{adjustbox}
\end{table} 

Using $5$ steps of GMRES to approximate the solution of the correction equation~(\ref{mjd}) in MJD method gives an eigenvalue approximation near the desired eigenvalue. The Ritz values converge to a spurious eigenvalue from $5^{th}$ restart onwards, but norm of the residual vectors reach the tolerance. The same behaviour is observed with its theoretically equivalent correction  equation~(\ref{mds}). With correction equations~(\ref{ojd}) and (\ref{jds}) in Jacobi-Davidson method, from $1^{st}$ restart onwards, the Ritz values stagnated near the desired eigenvalue. Convergence of Ritz values using correction equations (\ref{ojd}) and (\ref{mjd}) are shown in Figure~\ref{fig:fig7N}  . For correction equations (\ref{jds}) and (\ref{mds}), they are shown in Figure~\ref{fig:24eg1}. The dependence of $\log_{10}$ of norms of corresponding residual vectors versus restart numbers are shown in Figures~\ref{fig:8N} and \ref{fig:2dec24}.
\begin{figure}[!htb]
\begin{minipage}{0.4975\linewidth}
\begin{center} 
\includegraphics[width = 2in,height=2in]{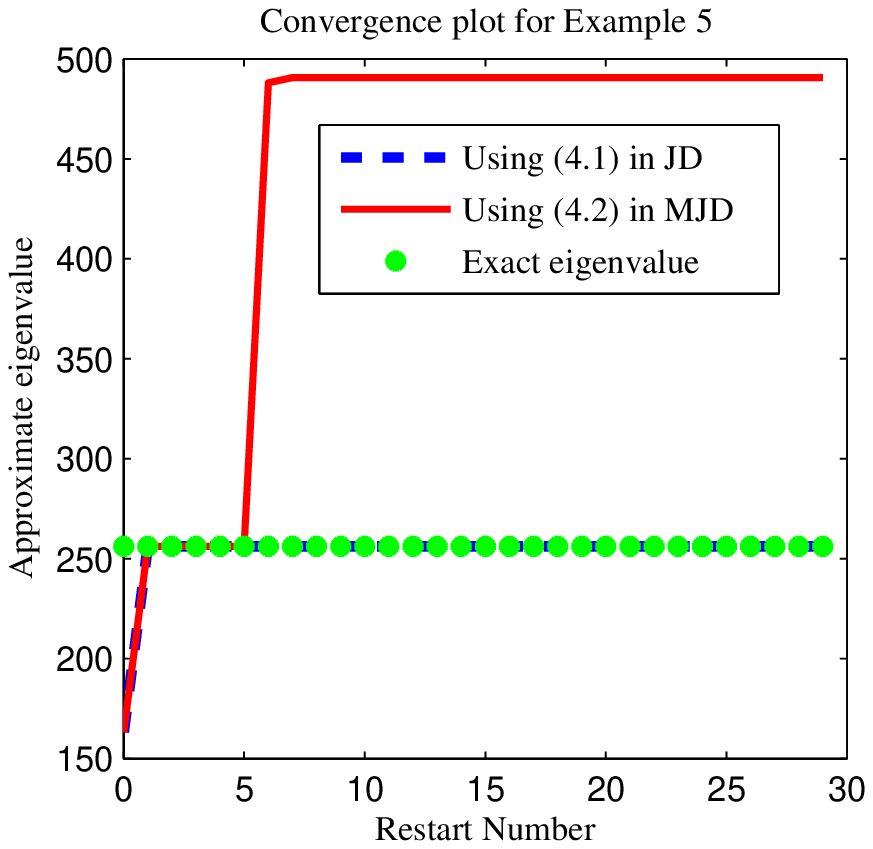}
\caption{\scriptsize \sl Restart number Vs Ritz values using \\correction equations (\ref{ojd}) and (\ref{mjd}) with \\approximate solution for subspace size $3$} 
\label{fig:fig7N} 
\end{center} 
\end{minipage}
\mbox{\hspace{0.5cm}}
\begin{minipage}{0.4975\linewidth}
\begin{center} 
\includegraphics[width = 2in,height=2in]{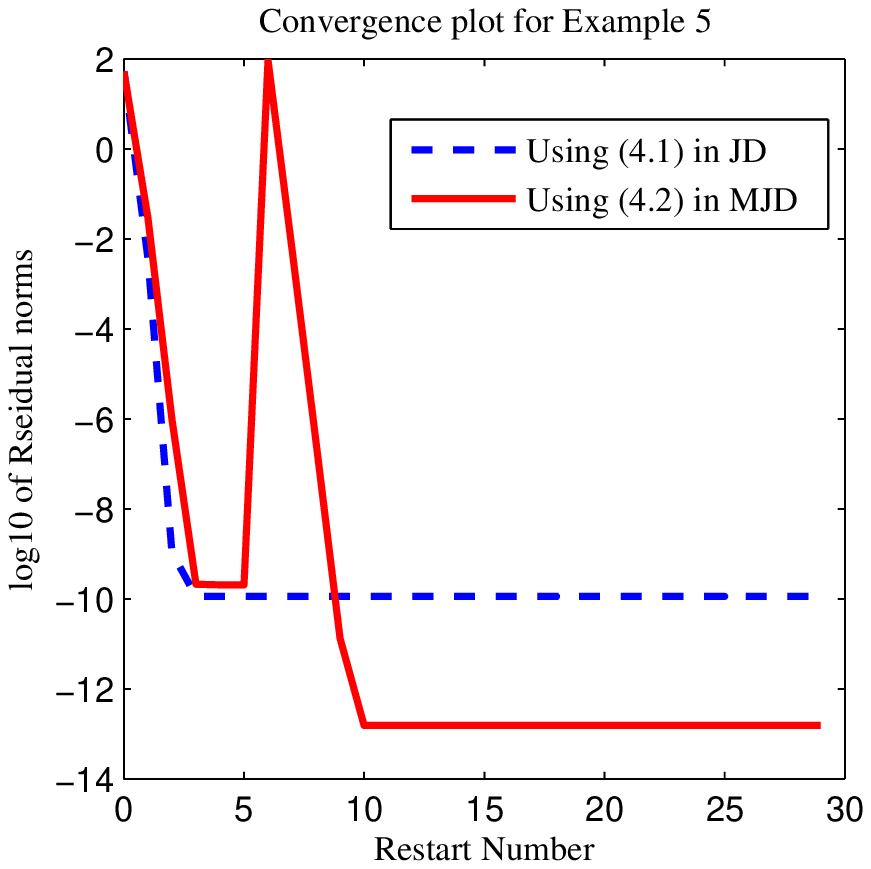}
\caption{\scriptsize \sl Restart number Vs $\log10(||res||_2)$ \\using correction equations (\ref{ojd}) and (\ref{mjd} ) \\with approximate solution for subspace size $3$} 
\label{fig:8N} 
\end{center} 
\end{minipage}
\end{figure}
\vs{-0.2cm}
\begin{figure}[!htb]
\begin{minipage}{0.4975\linewidth}
\begin{center} 
\includegraphics[width = 2in,height=2in]{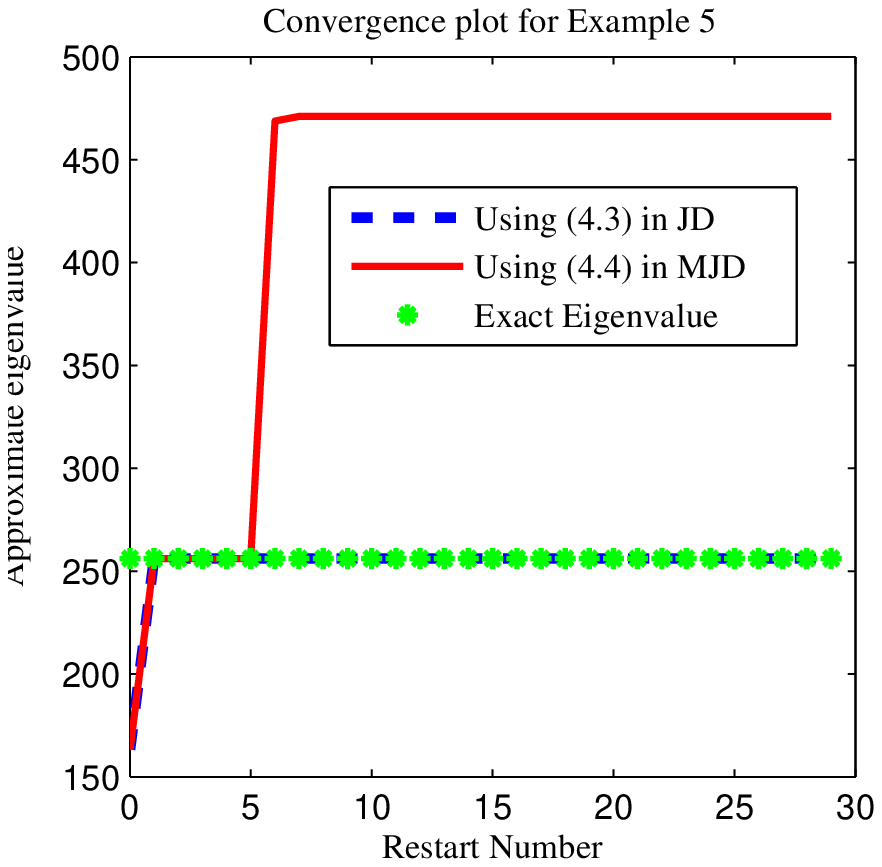}
\caption{\scriptsize \sl Restart number Vs Ritz values using \\correction equations (\ref{jds}) and (\ref{mds}) with \\approximate solution for subspace size $3$} 
\label{fig:24eg1} 
\end{center} 
\end{minipage}
\mbox{\hspace{0.5cm}}
\begin{minipage}{0.4975\linewidth}
\begin{center} 
\includegraphics[width = 2in,height=2in]{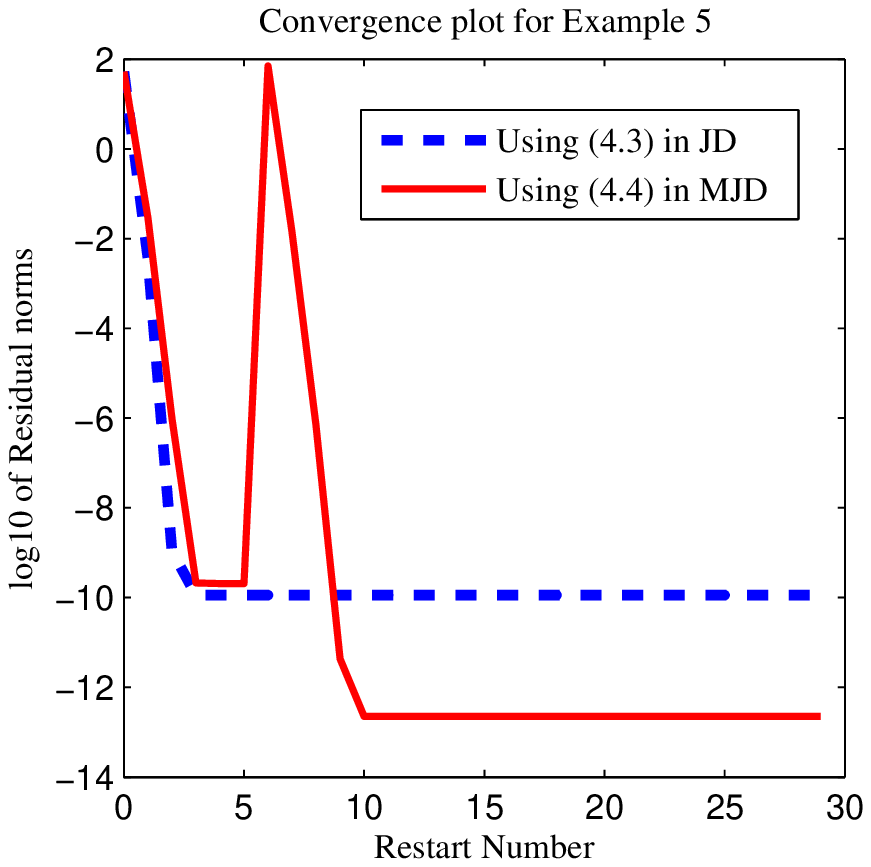}
\caption{\scriptsize \sl Restart number Vs $\log10(||res||_2)$ \\using correction equations (\ref{jds}) and (\ref{mds}) \\with approximate solution for subspace size $3$} 
\label{fig:2dec24} 
\end{center} 
\end{minipage}
\end{figure}

We also checked the performance of the proposed method by restarting the algorithm when the size of the subspace for extracting an eigen pair reached $4.$ The approximate solution of correction equations is obtained after $5$ steps of GMRES. Using the correction equation (\ref{ojd}), the eigenvalue approximation $ 2.561474561181778e+02$ is obtained at first restart, that is, without restart. After $4^{th}$ restart, it starts giving spurious eigenvalues whereas with the  correction equation (\ref{mjd}) the approximate eigenvalue obtained at $2^{nd}$ restart is found to be $2.561474561181784e+02$. Comparison of convergence of approximate eigenvalues in these cases is done in Figure~\ref{fig:fig4N}.Figure~\ref{fig:6N} reports the $\log_{10}$ of residual norms. 
\begin{figure}[!htb]
\begin{minipage}{0.4975\linewidth}
\begin{center} 
\includegraphics[width = 2in,height=2in]{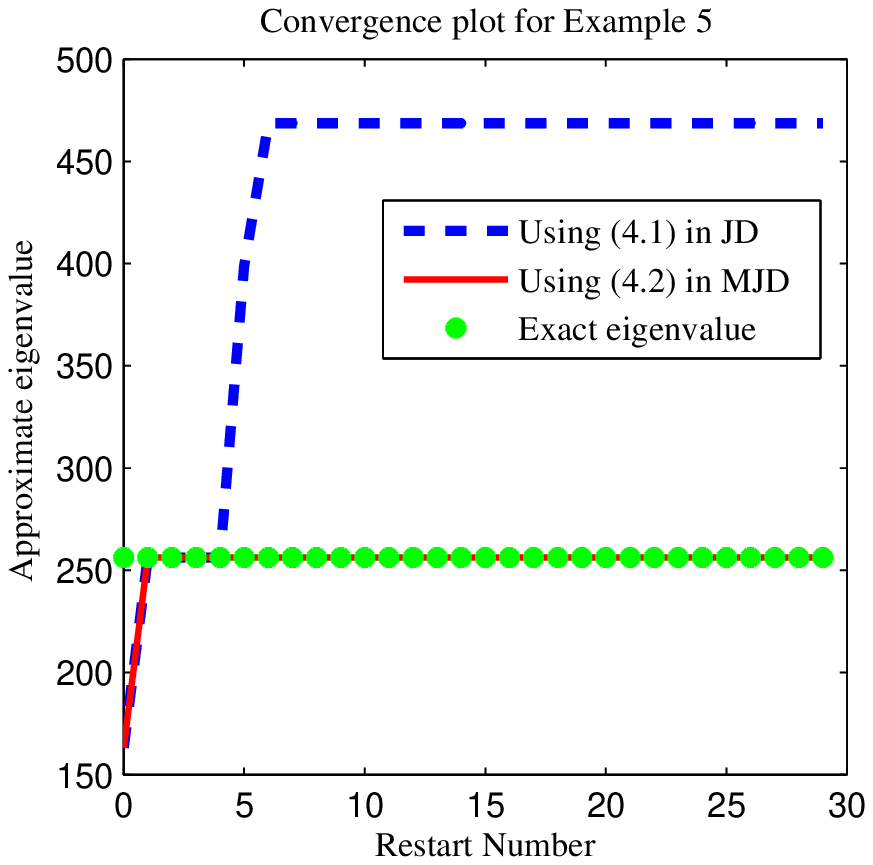}
\caption{\scriptsize \sl Restart number Vs Ritz values using \\correction equations (\ref{ojd}) and (\ref{mjd}) with \\approximate solution for subspace size $4$} 
\label{fig:fig4N} 
\end{center} 
\end{minipage}
\mbox{\hspace{0.5cm}}
\begin{minipage}{0.4975\linewidth}
\begin{center} 
\includegraphics[width = 2in,height=2in]{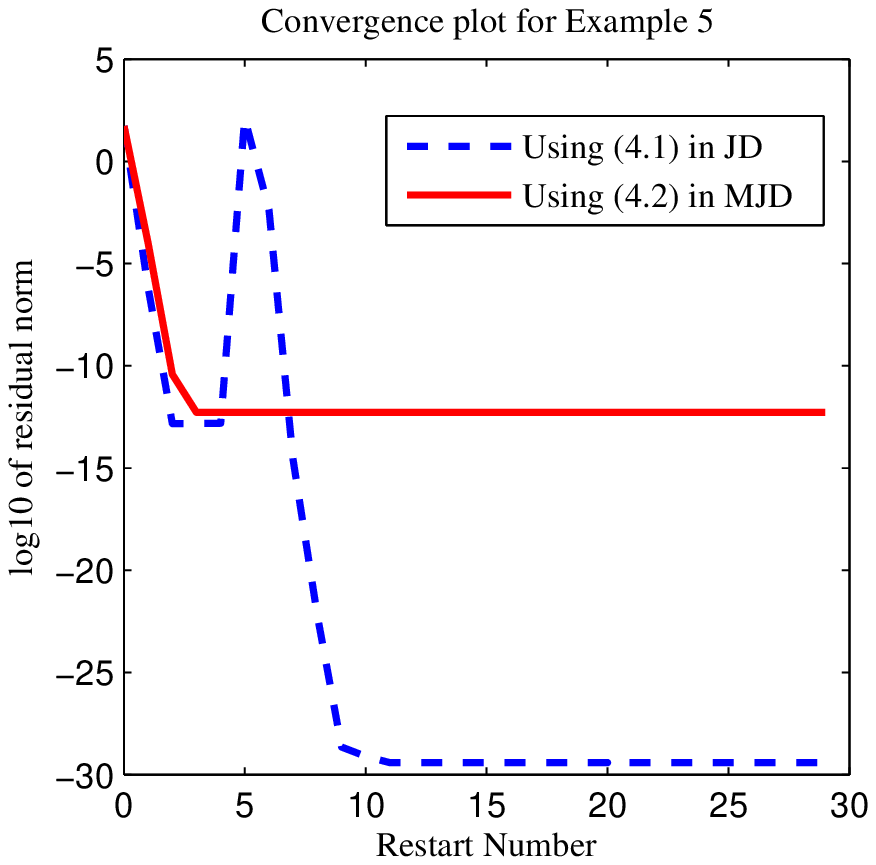}
\caption{\scriptsize\sl Restart number Vs $\log10(||res||_2)$ \\using correction equations (\ref{ojd}) and (\ref{mjd}) \\with approximate solution for subspace size $4$} 
\label{fig:6N} 
\end{center} 
\end{minipage}
\end{figure}

In a similar vein, using $5$ steps of GMRES for solving the correction equation~(\ref{jds}) and with subspace size $4,$ the approximation to the desired eigenvalue is obtained without restart, that is, when size of the subspace reaches $4.$ From $5^{th}$ restart, it starts giving a spurious eigenvalue. Whereas with the correction equation~(\ref{mds}), norm of residuals reached the tolerance at $2^{nd}$ restart and the approximation to eigenvalue is $2.561474561181783e+02.$ Results of comparison of convergence of Ritz values and norms of residuals with correction equations (\ref{jds}) and (\ref{mds}) are shown in Figure~\ref{fig:fig3N} and Figure~\ref{fig:5N}, respectively.
\begin{figure}[!htb]
\begin{minipage}{0.4975\linewidth}
\begin{center} 
\includegraphics[width = 2in,height=2in]{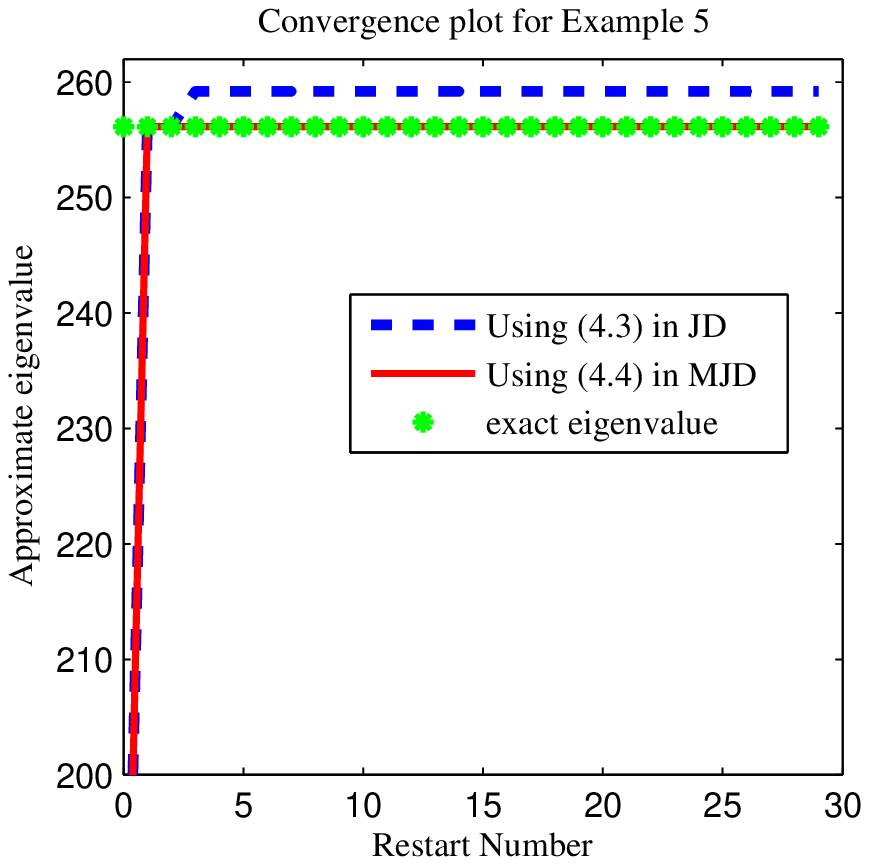}
\caption{\scriptsize \sl Restart number Vs Ritz values using \\correction equations (\ref{jds}) and (\ref{mds}) with \\approximate solution for subspace size $4$} 
\label{fig:fig3N} 
\end{center} 
\end{minipage}
\mbox{\hspace{0.5cm}}
\begin{minipage}{0.4975\linewidth}
\begin{center} 
\includegraphics[width = 2in,height=2in]{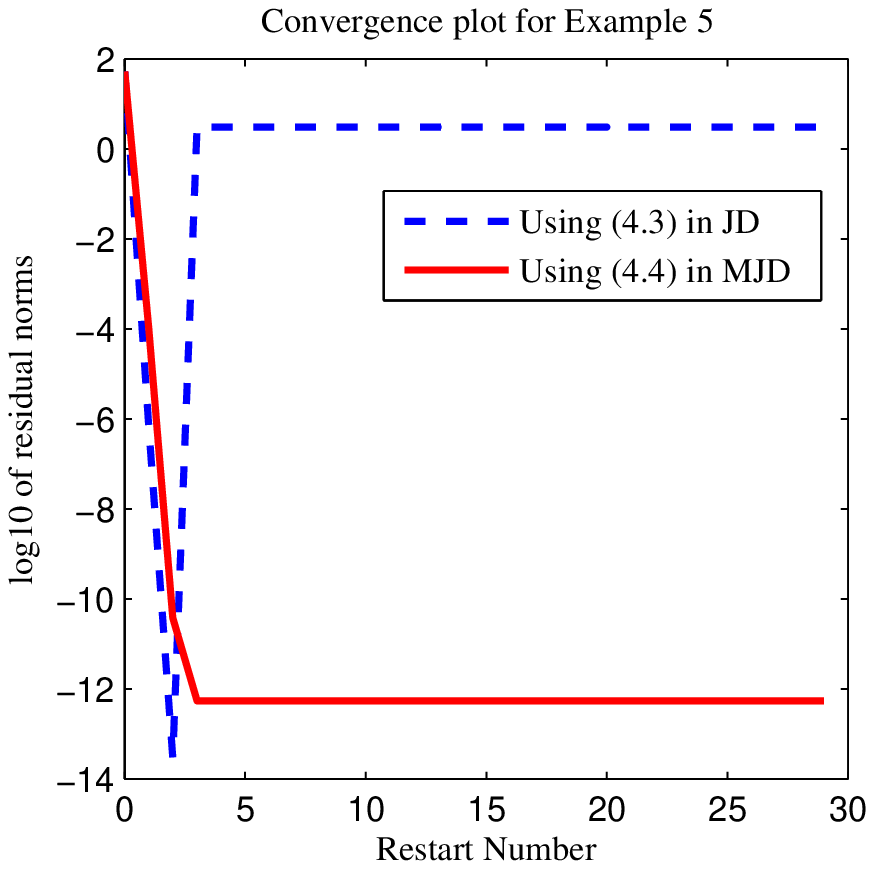}
\caption{\scriptsize \sl Restart number Vs $\log10(||res||_2)$ \\using correction equations (\ref{jds}) and (\ref{mds}) \\with approximate solution for subspace size $4$} 
\label{fig:5N} 
\end{center} 
\end{minipage}
\end{figure}
\begin{example}\label{eg2}\relax
Let the matrix $A =QTQ,$ where $T$ is a tridiagonal matrix of order $100$ with super-diagonal and sub-diagonal elements as $-1,$ and diagonal elements $1$, and $Q$ is the 
 Householder transformation of order $100$  with Householder vector $h$ having entries $h_i = \sqrt{i+.45}$ for $i = 1,2,....100$. The matrix $A$ is not diagonally dominant. We take each entry in the Initial vector as $1$. The largest eigenvalue is required. Matlab command `eig' produces an approximation to the largest eigenvalue as $ 3.999032564583972e+00$. See Example 2 in \cite{slei}.
\end{example}

We first computed eigenvalue approximations by applying Rayleigh-Ritz projection over subspaces of dimension ranging from $3$ to $20$. Using Gaussian elimination for solving the correction equation (\ref{jds}), a good approximation to the desired eigenvalue is obtained at $ 8^{th}$ restart for the subspace of dimension $15,$ where the Ritz value and residual norm are $3.999032566618813e+00$ and $2.034836893782876e-09,$ respectively.  The same accuracy to the desired eigenvalue is also obtained for the subspace of dimension $15$, when the correction equation (\ref{mds}) is  solved using Gaussian elimination. Figures ~\ref{fig:fig9N} and \ref{fig:10N}, respectively, show the comparison of Ritz values and residual norms in these two cases. \vs{2cm}

\begin{figure}[!htb]
\begin{minipage}{0.4975\linewidth}
\begin{center} 
\includegraphics[width = 2in,height=2in]{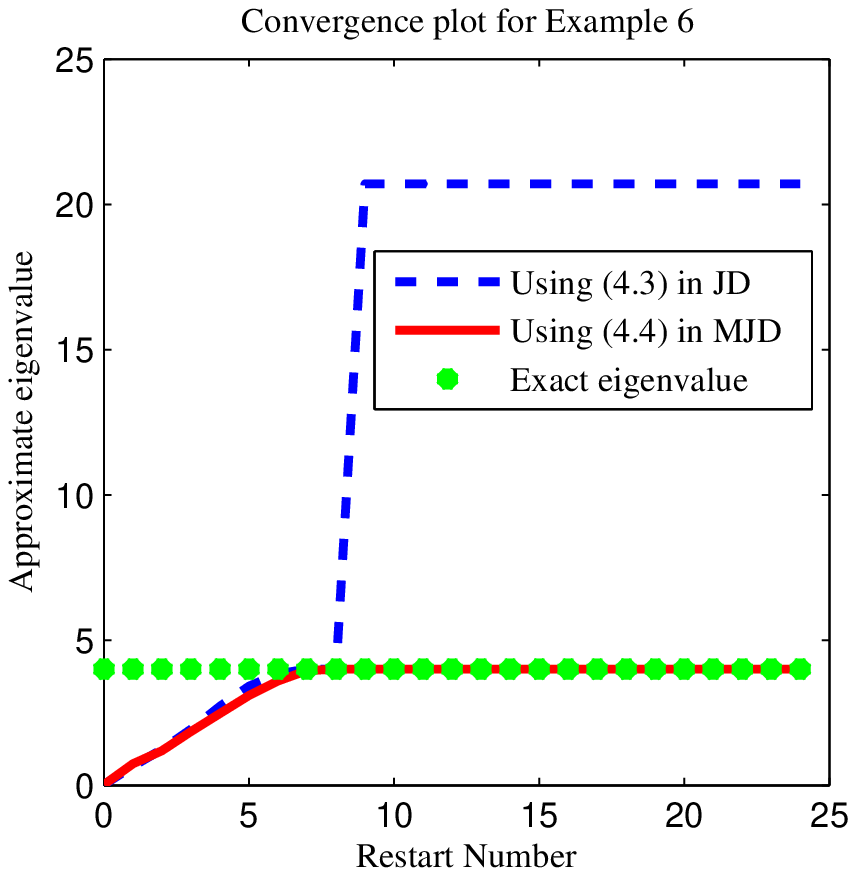}
\caption{\scriptsize \sl Restart number Vs Ritz values using \\correction equations (\ref{jds}) and (\ref{mds}) with Gaus-\\sian elimination solution for subspace size $15$} 
\label{fig:fig9N} 
\end{center} 
\end{minipage}
\mbox{\hspace{0.5cm}}
\begin{minipage}{0.4975\linewidth}
\begin{center} 
\includegraphics[width = 2in,height=2in]{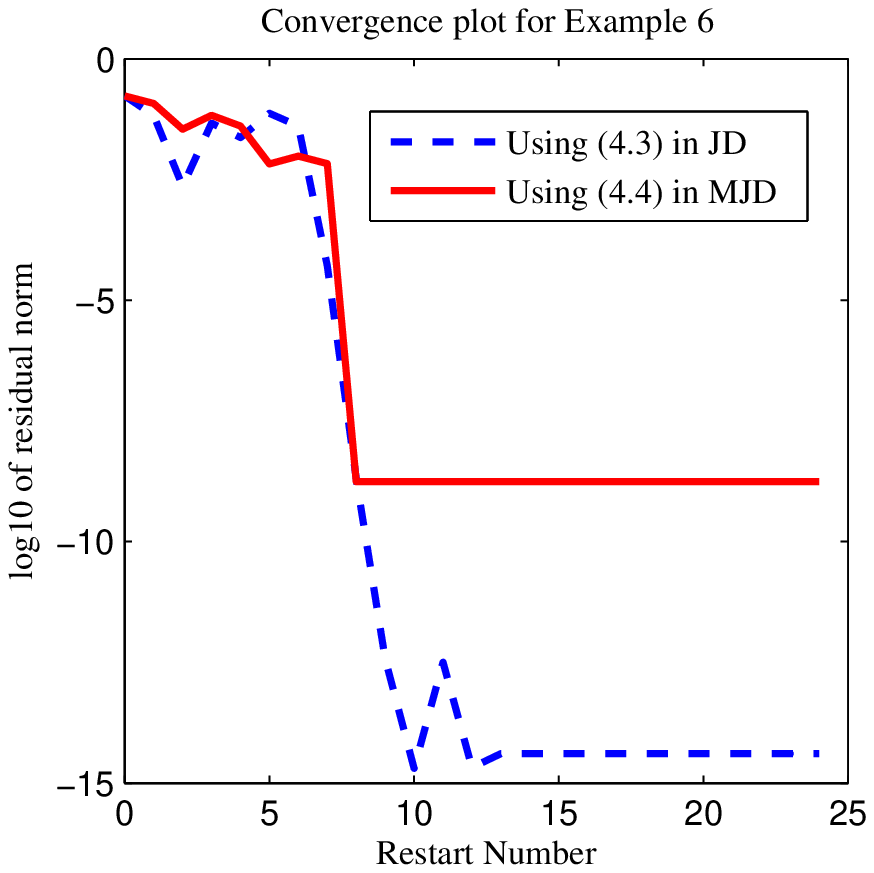}
\caption{\scriptsize Restart number Vs $\log10(||res||_2)$ \\using correction equations (\ref{jds}) and (\ref{mds}) with \\Gaussian elimination solution for subspace size $15$} 
\label{fig:10N} 
\end{center} 
\end{minipage}
\end{figure}
For subspaces of dimensions $11$ and $14$ in Rayleigh-Ritz projection, an accuracy upto machine precision is obtained for the desired eigenvalue at $13^{th}$ and $10^{th}$ iterations, respectively. The results are shown in Figure~\ref{fig:fig11N} and the convergence of residual norms associated with this is shown in Figure \ref{fig:12N}. But, eigenvalue approximations obtained using the correction equation (\ref{jds}) are not accurate upto machine precision for a subspace of any size.

\begin{figure}[!htb]
\begin{minipage}{0.4975\linewidth}
\begin{center} 
\includegraphics[width = 2in,height=2in]{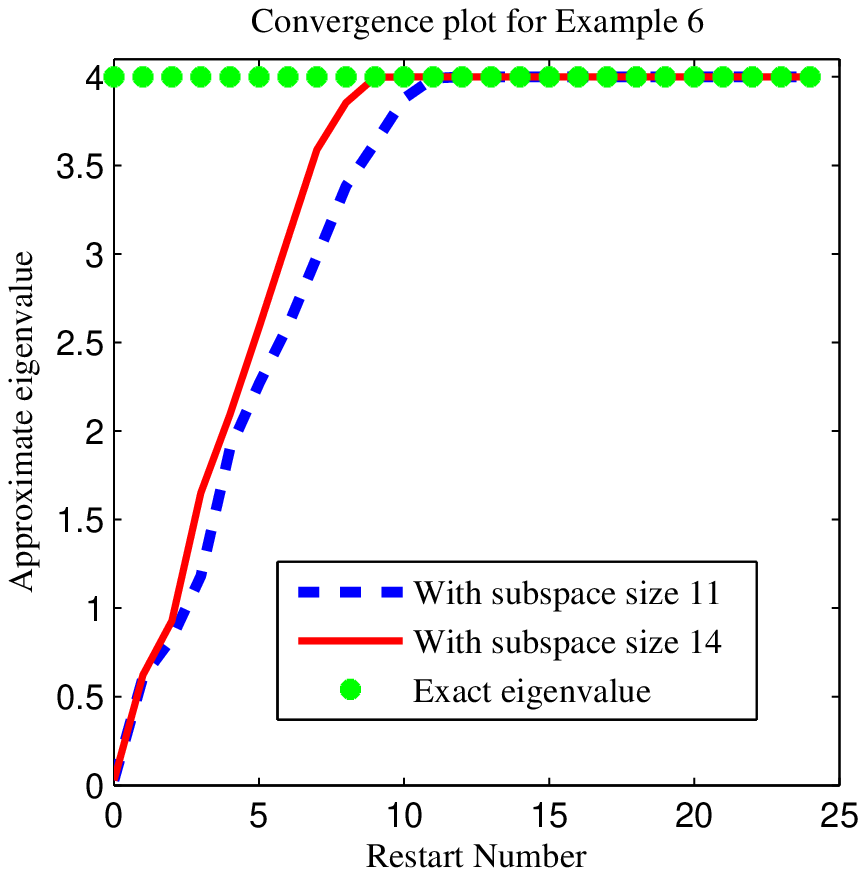}
\caption{\scriptsize \sl Restart number Vs Ritz values using \\correction equation(\ref{mds}) with Gaussian elimination \\solution for subspace sizes $11$ and $14$} 
\label{fig:fig11N} 
\end{center} 
\end{minipage}
\mbox{\hspace{0.5cm}}
\begin{minipage}{0.4975\linewidth}
\begin{center} 
\includegraphics[width = 2in,height=2in]{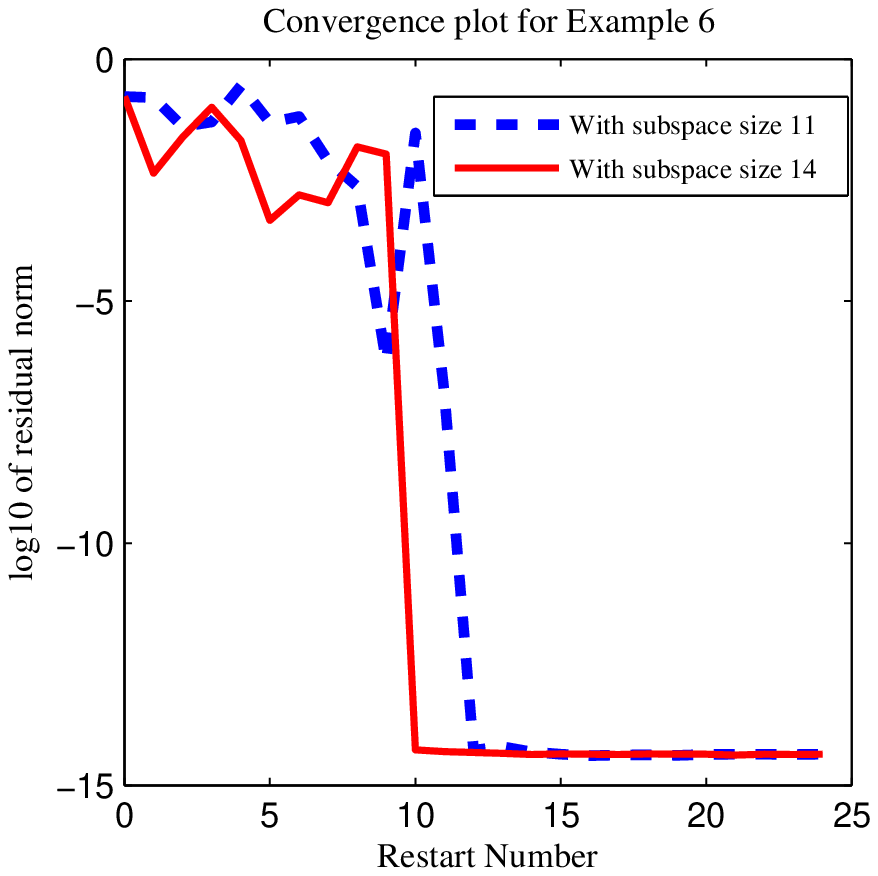}
\caption{\scriptsize \sl Restart number Vs $\log10(||res||_2)$ \\using correction equation(\ref{mds}) with Gaussian \\elimination solution for subspace size $11$ and $14$} 
\label{fig:12N} 
\end{center} 
\end{minipage}
\end{figure}

\begin{figure}[!htb]
\begin{minipage}{0.4975\linewidth}
\begin{center} 
\includegraphics[width = 2in,height=2in]{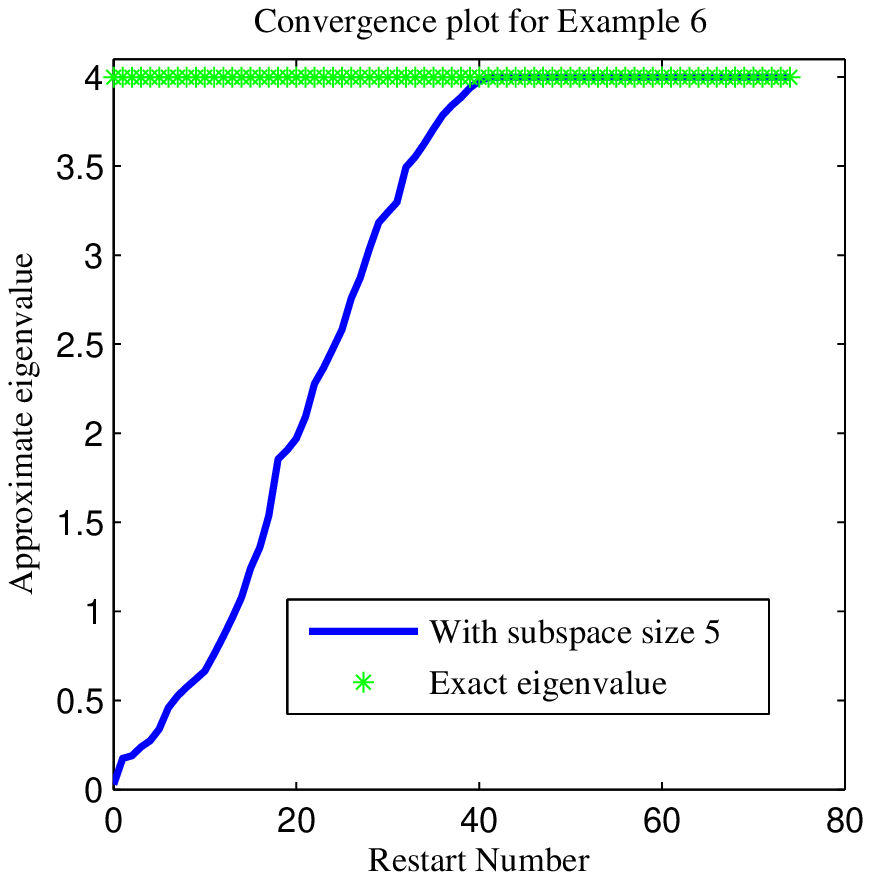}
\caption{\scriptsize \sl Restart number Vs Ritz values using \\correction equation(\ref{mds}) with Gaussian elimination \\solution for subspace size $5$} 
\label{fig:fig13N} 
\end{center} 
\end{minipage}
\mbox{\hspace{0.5cm}}
\begin{minipage}{0.4975\linewidth}
\begin{center} 
\includegraphics[width = 2in,height=2in]{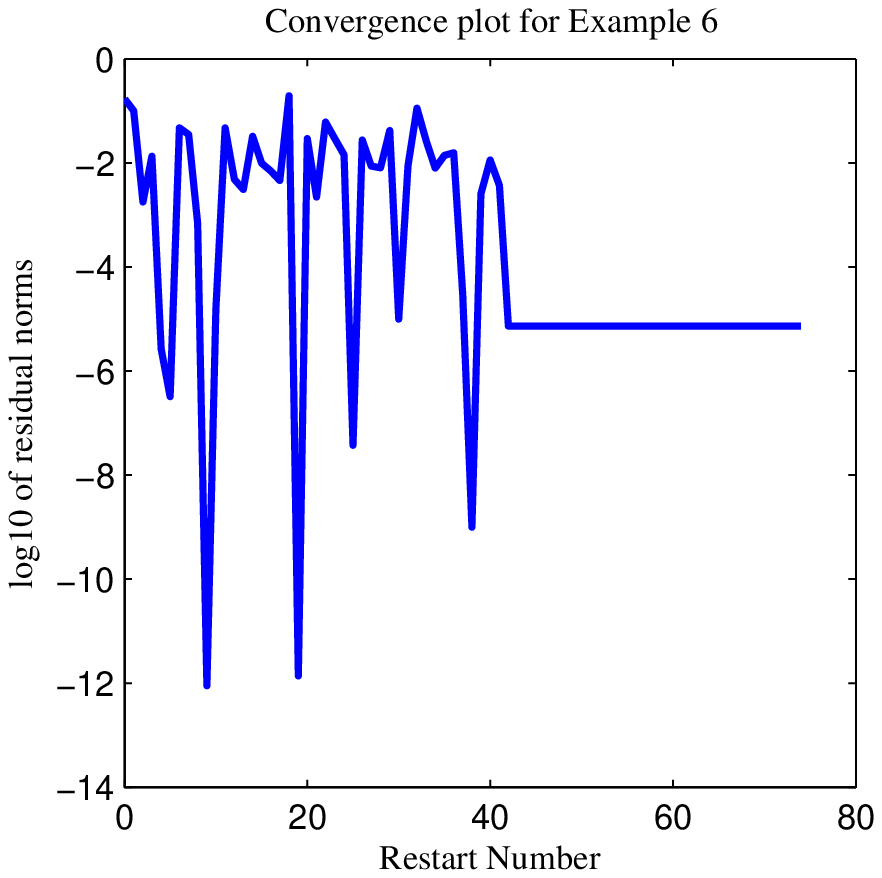}
\caption{\scriptsize \sl Restart number Vs $\log10(||res||_2)$ \\using correction equation(\ref{mds} ) with Gaussian \\elimination solution for subspace size $5$ } 
\label{fig:14N} 
\end{center} 
\end{minipage}
\end{figure}

As mentioned earlier, the stagnation phenomenon occurs with the correction equation (\ref{mds}) when the subspace is of size  $5,$ from $45^{th}$ iteration onwards with Ritz value $ \theta = 3.999039854574475e+00.$ This is accurate upto $5$ decimal places with the residual norm as $7.289990498570303e-06  $. In this case $$\|\big(((A-\theta I)^\ast (A-\theta I)-\|residual\|^2 I\big)x\| =  5.314441250548710e-11$$ where $x$ is a refined Ritz vector. A right singular vector of $A-\theta I$  corresponding to the singular value $(7.289990498570303e-06)^2$ is obtained. Figures \ref{fig:fig13N} and \ref{fig:14N} show the convergence of Ritz values and residual norms.
\begin{figure}[!htb]
\begin{minipage}{0.4975\linewidth}
\begin{center} 
\includegraphics[width = 2in,height=2in]{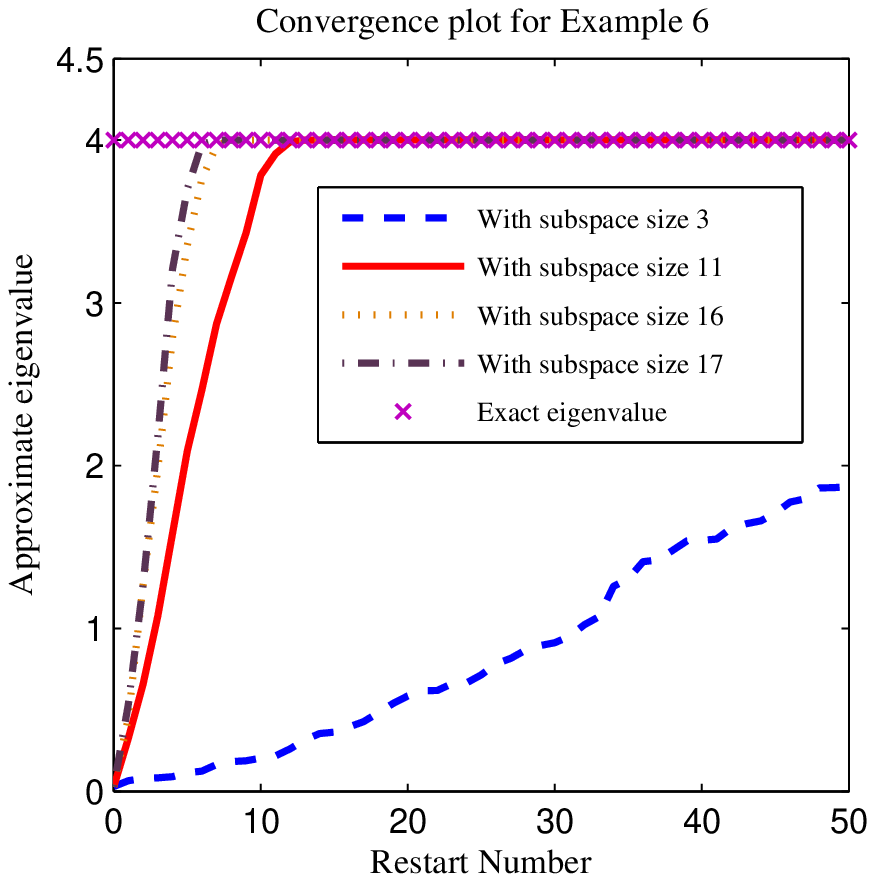}
\caption{\scriptsize \sl Restart number Vs Ritz values using \\correction equation(\ref{ojd}) with Gaussian elimination \\solution for subspace sizes $3$,$11$,$16$ and $17$} 
\label{fig:fig15N} 
\end{center} 
\end{minipage}
\mbox{\hspace{0.5cm}}
\begin{minipage}{0.4975\linewidth}
\begin{center} 
\includegraphics[width = 2in,height=2in]{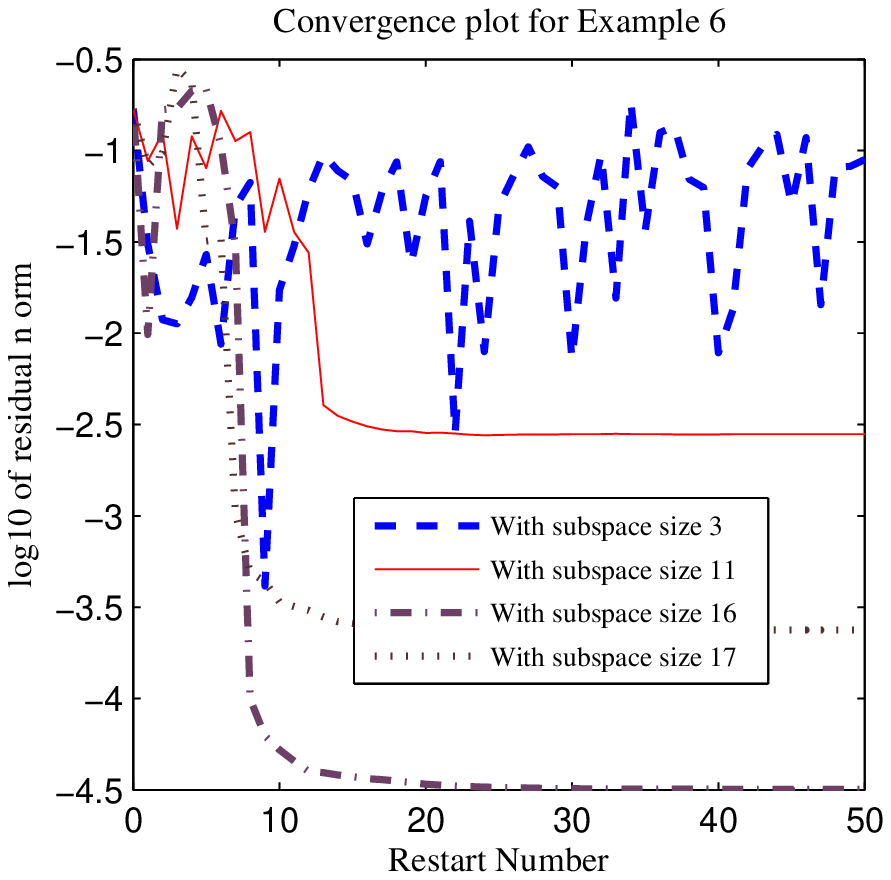}
\caption{\scriptsize \sl Restart number Vs $\log10(||res||_2)$ \\using correction equation(\ref{ojd} ) with Gaussian \\elimination solution for subspace size $3$,$11$,$16$ and $17$} 
\label{fig:16N} 
\end{center} 
\end{minipage}
\end{figure}

Matrices in the correction equation (\ref{ojd}) are found to be ill-conditioned during computation, when exact solutions are required except for subspaces of sizes $3,11,16$ and $17.$ In these exceptional cases, the Ritz values are found to be accurate up to $4$ decimal places. Comparison of convergence behaviour of Ritz values and residual norms  are shown in Figures \ref{fig:fig15N}-\ref{fig:16N}, for subspaces of sizes $3,11,16$ and $17$.

Using Gaussian elimination for solving the correction equation (\ref{mjd}), the exact eigenvalue is obtained at $16^{th}$ iteration, when the subspace of size $10$ is used in the Rayleigh-Ritz projection. Matrices in the correction equation (\ref{mjd}) become almost singular only for subspaces of size $6$ and $11$. Convergence of Ritz values and residual norms with various subspace sizes, using the correction equation (\ref{mjd}) are shown in Figures \ref{fig:fig17N}-\ref{fig:18N}. 

\begin{figure}[!htb]
\begin{minipage}{0.4975\linewidth}
\begin{center} 
\includegraphics[width = 2in,height=2in]{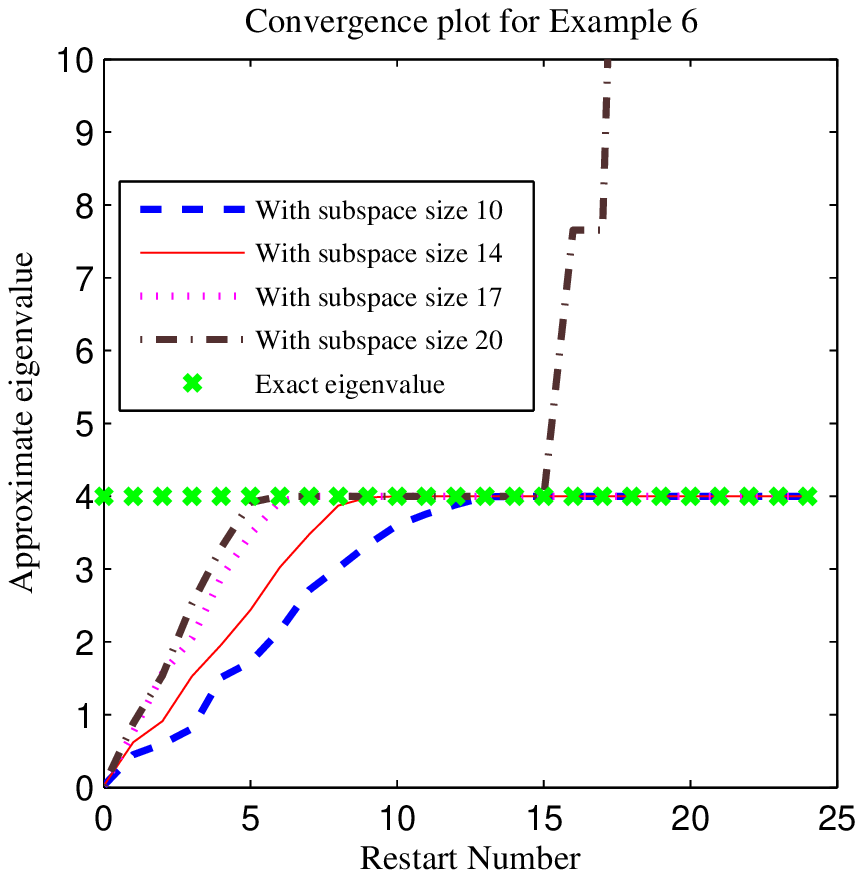}
\caption{\scriptsize \sl Restart number Vs Ritz values using \\correction equation(\ref{mjd}) with Gaussian elimination \\solution for subspace sizes $10,14, 17$ and $20$} 
\label{fig:fig17N} 
\end{center} 
\end{minipage}
\mbox{\hspace{0.5cm}}
\begin{minipage}{0.4975\linewidth}
\begin{center} 
\includegraphics[width = 2in,height=2in]{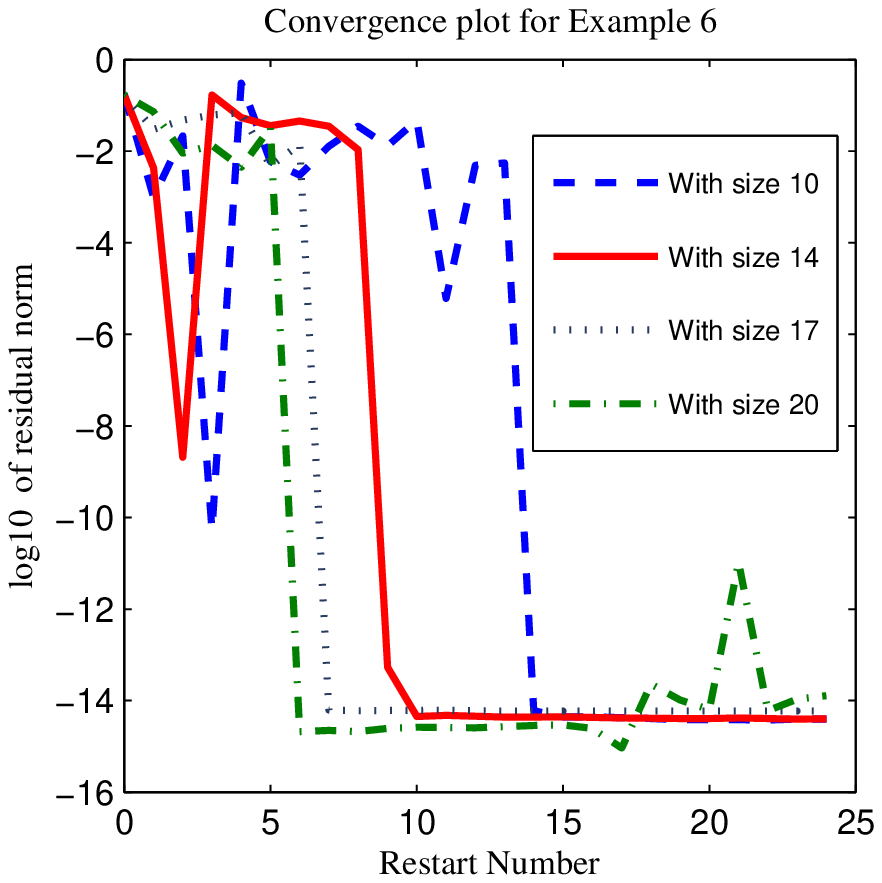}
\caption{\scriptsize \sl Restart number Vs $\log10(||res||_2)$ \\using correction equation(\ref{mjd} ) with Gaussian \\elimination solution for subspace size $10,14, 17$ and $20$} 
\label{fig:18N} 
\end{center} 
\end{minipage}
\end{figure}

When an approximate solution of Equations (\ref{ojd}), (\ref{jds}) and (\ref{mjd}), (\ref{mds}) are obtained using $5$ steps of GMRES method, an eigenvalue approximation obtained is found to be accurate upto $4$ decimal places. It may be explained as follows. A possible explanation is that due to the correction equations in Jacobi Davidson method, the difference between Ritz values in two consecutive iterations is very high whereas with new correction equations, the difference is low leading to slow convergence.

\section{Conclusion and future work}
In this paper, we have proposed a modification to the subspace expansion phase in Jacobi-Davidson method for computing approximate eigenvalues of a large sparse matrix. The modification uses the heuristic of least squares. Theoretically, the modification has the advantage that it is still applicable to the cases when the correction equation obtained in Jacobi-Davidson method results in a singular system matrix. Further, the modified method is theoretically equivalent to the Alternating Rayleigh quotient iteration, which converges globally and proposed by B.N. Parlett in  \cite{grqi}. To check whether the modification performs well computationally, we have considered many bench mark examples. It is observed that the over all performance of the modified algorithm is well comparable with the Jacobi-Davidson method. Along with the required eigenvalue, approximations to other eigenvalues are also obtained. In case the proposed modified method exhibits slow convergence, compared to Jacobi-Davidson method, it gives good approximation to many eigenvalues including the desired one. The slow convergence is attributed to the clustering of eigenvalues near the current approximation. While Jacobi-Davidson method jumps away from this cluster resulting in an approximation to a different eigenvalue than the desired one, the proposed method approaches slowly towards the desired eigenvalue. It has been observed that when stagnation occurs in the modified method, an approximation to the right singular vector of a matrix $(A-\theta I)$ is obtained. This is observed in Example~\ref{eg2}, where the norm of the residual vector is a singular value. When the norm of the residual coincides with a smallest singular value, the obtained vector is likely to be a good approximation to an eigenvector of the matrix, associated with an approximate eigenvalue $\theta.$ The theory about this coincidence is yet to be developed. 

\end{document}